\documentclass[11pt]{amsart}
\usepackage[latin1]{inputenc}
\usepackage[usenames]{color}
\usepackage{amsfonts}
\usepackage{amsthm}
\usepackage{amsmath}
\usepackage{amssymb}
\usepackage{amscd}
\usepackage{multicol}
\usepackage{tabularx}

\theoremstyle{definition}
\newtheorem{thm}{Theorem}[subsection]
\newtheorem{prop}[thm]{Proposition}
\newtheorem{cor}[thm]{Corollary}

\newtheorem{lem}[thm]{Lemma}
\newtheorem{defn}[thm]{Definition}
\newtheorem{ex}[thm]{Example}
\newtheorem*{rem}{Remark}

\numberwithin{equation}{subsection}

\def\bb#1{\text{$\mathbb{#1}$}}
\def\blb#1{\text{$\mathbb{#1}$}} 
\def\cal#1{\text{${\mathcal{#1}}$}}

\def\lie#1{\mathfrak{#1}}
\def\tlie#1{\tilde{\mathfrak{#1}}}
\def\hlie#1{\hat{\mathfrak{#1}}} 

\def\biq#1^#2{\text{$\left[\begin{matrix} #2\\#1\end{matrix}\right]_q$}}
\def\bis#1^#2{\text{$\left[\begin{matrix} #2\\#1\end{matrix}\right]_{\zeta_i}$}}

\def\bik#1^#2{\text{$\left[\begin{matrix} k_i; #2\\#1\end{matrix}\right]$}}

\def\uq#1{\text{$U_q(\lie #1)$}}
\def\uqr#1^#2{\text{$U_q^{#2}(\lie #1)$}}
\def\uqres#1{\text{$U_\mathbb A(\lie #1)$}}
\def\us_#1#2{\text{$U_{#1}(\lie #2)$}}
\def\dus#1^#2{\text{$\dot{U}_{\zeta}^{#2}(\lie #1)$}}
\def\usres#1{\text{$U_{\zeta}(\lie #1)$}}

\def\uqt#1{\text{$U_q(\tlie #1)$}}
\def\uqtr#1^#2{\text{$U_q^{#2}(\tlie #1)$}}
\def\uqtres#1{\text{$U_\mathbb A(\tlie #1)$}}
\def\ust_#1#2{\text{$U_{#1}(\tlie #2)$}}
\def\ustres#1{\text{$U_{\zeta}(\tlie #1)$}}
\def\usfin#1{\text{$U_{\zeta}^{\rm fin}(\lie #1)$}}
\def\ustfin#1{\text{$U_{\zeta}^{\rm fin}(\tlie #1)$}}
\def\ustf_#1#2{\text{$U_{#1}^{\rm fin}(\tlie #2)$}}
\def\usf_#1#2{\text{$U_{#1}^{\rm fin}(\lie #2)$}}

\def\dust#1^#2{\text{$\dot{U}_{\zeta}^{#2}(\tlie #1)$}}

\def\uqh#1{\text{$U_q(\tlie #1)$ }}
\def\uqhr#1^#2{\text{$U_q^{#2}(\tlie #1)$ }}
\def\ush#1^#2{\text{$U_{\zeta}^{#2}(\tlie #1)$ }}
\def\dush#1^#2{\text{$\dot{U}_{\zeta}^{#2}(\tlie #1)$ }}

\def\gb#1{{\mbox{\boldmath $#1$}}}
\def\gbr#1{{\mbox{\boldmath ${\rm #1}$}}}
\def\bgb#1{{\mbox{$\overline{\gb #1}$}}}
\def\wtl{{\rm wt}_\ell}
\def\wt{{\rm wt}}
\def\ch{{\rm ch}}
\def\chl{{\rm qch}}

\def\opl_#1^#2{\text{\scriptsize$\bigoplus\limits_{\text{\footnotesize$#1$}}^{\text{\footnotesize$#2$}}$}}
\def\otm_#1^#2{\text{\scriptsize$\bigotimes\limits_{\text{\footnotesize$#1$}}^{\text{\footnotesize$#2$}}$}}
\def\wcal#1{{\mbox{$\widetilde{\cal #1}$}}}

\def\z{\bb Z}
\def\zp{\text{$\bb Z_{>0}$}}
\def\zpz{\text{$\bb Z_{\ge 0}$}}
\def\znn{\text{$\bb Z_{\ge 0}$}}

\def\zs#1{\text{$\bb Z_{#1}$}}

\def\tqbinom#1#2{\text{$\left[\begin{smallmatrix} #1\\#2\end{smallmatrix}\right]$}}

\begin{document}

\flushbottom

\title[]{Tensor Products, Characters, and Blocks of\\  Finite-Dimensional Representations of\\ Quantum Affine Algebras at Roots of Unity}

\author{Dijana Jakeli\'c and Adriano Moura}
\address{Department of Mathematics and Statistics, University of North Carolina Wilmington, \vspace{-12pt}}
\address{601 S. College Road, Wilmington, NC, 28401-5970}
\email{jakelicd@uncw.edu}

\address{UNICAMP - IMECC, Campinas - SP, 13083-970, Brazil} \email{aamoura@ime.unicamp.br}

\maketitle
\centerline{\small{
\begin{minipage}{350pt}
{\bf Abstract:} We establish several results concerning tensor products, q-characters, and the block decomposition of the category of finite-dimensional representations of quantum affine algebras  in the root of unity setting. In the generic case, a Weyl module is isomorphic to a tensor product of fundamental representations and this isomorphism was essential for establishing the block decomposition theorem. This is no longer true in the root of unity setting. We overcome the lack of such a tool by utilizing results on specialization of modules. Furthermore, we establish a sufficient condition for a Weyl module to be a tensor product of fundamental representations and prove that this condition is also necessary when the underlying simple Lie algebra is $\lie{sl}_2$. We also study the braid group invariance of q-characters of fundamental representations.
\end{minipage}}}

\section*{Introduction}

Let $\lie g$ be a finite-dimensional simple Lie algebra over the complex numbers and $\tlie g=\lie g\otimes\mathbb C[t,t^{-1}]$ its associated loop algebra. Following Drinfeld and Jimbo, one can consider the quantum groups $U_q(\lie g)$ and $U_q(\tlie g)$ which are Hopf algebras over the field $\mathbb C(q)$ of rational functions in the indeterminate $q$. The latter is most often called a quantum affine algebra.  Using $\mathbb C[q,q^{-1}]$-forms of \uq g and \uqt g, one defines specializations of the quantum groups at $q=\xi\in\mathbb C\backslash\{0\}$. If $\xi$ is generic, i.e., not a root of unity, the ``general behavior'' of the resulting algebras does not depend on $\xi$ or on the chosen form. However, if $\xi$ is a root of unity, the resulting algebras depend drastically on the chosen form, as well as on $\xi$.

In this paper, we focus on the algebras obtained from Lusztig's form which is generated by the divided powers of the quantum Chevalley generators of \uq g or \uqt g.  The resulting algebras will be denoted by \us_\xi g and \ust_\xi g, respectively.
We study the category $\wcal C_\xi$ of (type 1) finite-dimensional representations of \ust_\xi g with special attention to the root of unity setting, where the order of the root of unity is assumed to be odd and relatively prime to the lacing number of $\lie g$. The first results regarding this category in the root of unity context were obtained by Chari and Pressley in \cite{cproot} (for results concerning De Concini-Kac's specialization see \cite{beka,dchere} and references therein). Our goal is to establish several results concerning tensor products, q-characters,
and the block decomposition of $\wcal C_\xi$ for the case that $\xi$ is a root of unity.

The first of the main results of the present paper (Theorem \ref{t:hlwtensor}) is a generalization of the  main result of \cite{vcbraid} which gives a sufficient condition for a tensor product of simple objects of $\wcal C_\xi$ to be a quotient of the appropriate Weyl module. In the generic case, this result combined with results of Beck and Nakajima \cite{becnak:2sc} is used in \cite{cmq} to prove that the Weyl modules are isomorphic to certain tensor products of fundamental modules. It is known that in the root of unity setting this is not the case \cite[Remark 7.18]{vava:standard}. However, Theorem \ref{t:hlwtensor} and the results of \cite{becnak:2sc} (see also \cite{jm:weyl}) can be used to obtain a  sufficient combinatorial condition for this to be true (Corollary \ref{c:wtfundrgen}). Moreover, if $\lie g=\lie{sl}_2$, this condition is also necessary (Corollary \ref{c:wtfundr}). The combinatorial conditions appearing in all of these results are defined using certain factorizations of polynomials and the braid group action on the $\ell$-weight lattice (see Section \ref{s:lattices}). These factorizations of polynomials are motivated by the classification of the simple objects of $\wcal C_\xi$ when $\lie g=\lie{sl}_2$ in terms of tensor products of evaluation modules \cite{cp:qaa,cproot}. The braid group action is that of \cite{vcbraid} and the combinatorial condition we use in the statement of Theorem \ref{t:hlwtensor} is an extension to the root of unity setting of the one used in the main result of \cite{vcbraid}. The combinatorial condition used in the statements of Corollaries \ref{c:wtfundr} and \ref{c:wtfundrgen} appears to not have been defined before.

Our next goal is to describe the block decomposition of $\wcal C_\xi$. This was first studied in \cite{em} where it was proved that the blocks are parameterized by the so-called elliptic characters which are elements of the quotient of the $\ell$-weight lattice of \ust_\xi g by the $\ell$-root lattice. However, in that paper $\xi$ was assumed to be a nonzero complex number satisfying $|\xi|\ne 1$. The reason behind this assumption was that the original approach of \cite{em} used analytic properties of the action of the R-matrix of \ust_\xi g which could be guaranteed only for such $\xi$. A new approach, in connection with the theory of q-characters, was developed in \cite{cmq} where it was used to describe the blocks of $\wcal C_\xi$ for all generic $\xi$. We will mostly follow this later approach to describe the blocks of $\wcal C_\xi$ in the case that $\xi$ is a root of unity. However, our results concerning tensor products show that one of the main techniques used in the generic context is not available in the root of unity one.
Namely, an essential tool for proving that two objects having the same elliptic character must be in the same block was a corollary of \cite[Theorem 4.4]{vcbraid} stating that a tensor product of fundamental representations can always be reordered in such a way that it is a quotient of the corresponding Weyl module. In fact, as mentioned above, it was shown in \cite{cmq} that such reordered tensor products are isomorphic to the corresponding Weyl module thus proving that every Weyl module is isomorphic to a tensor product of fundamental representations. As discussed above, the latter statement, as well as the aforementioned corollary, is false in the root of unity setting.
To overcome this issue, we consider specialization of modules. Namely, we prove that the irreducible quotient of the specialization of an irreducible constituent of a Weyl module is an irreducible constituent of the specialization of that Weyl module. Using this fact and combinatorial arguments, we ``lift'' simple objects having the same elliptic character from the root of unity to the formal $q$ context in such a way that the modules thus obtained have the same elliptic character again. We then use the results of \cite{cmq} showing that these lifted modules are linked by a certain sequence of Weyl modules. By specializing this sequence back to the root of unity context we are  able to show that the starting two modules must have been in the same block. We remark that this proof also works for $\xi=1$, in which case $\wcal C_\xi$ is the category of finite-dimensional $\tlie g$-modules. The blocks for $\xi=1$ were described in \cite{cms} using a different approach (for showing that two objects with the same elliptic character must be in the same block). Hence, our proof is also an alternate one to the proof of the block decomposition theorem in the classical context. In fact, the proof works for any generic value of $\xi$ and, therefore, it can be thought of as a uniform proof.

Our last goal is to show that the main result of \cite{cmq} regarding the braid group invariance of the q-characters of fundamental representations holds in the root of unity context as well. The theory of q-characters, initiated by Frenkel and Reshetikhin in \cite{fr}, is one of the most interesting topics of the finite-dimensional representation theory of quantum affine algebras. In the root of unity context they were first considered by Frenkel and Mukhin in \cite{fmroot}. Frenkel and Mukhin also developed the first algorithm for computing the q-character of certain simple modules which is now known as the Frenkel-Mukhin algorithm \cite{fmq}. There are several other papers dedicated to the study of q-characters of the finite-dimensional irreducible $U_\xi(\tlie g)$-modules and other important subclasses of representations such as standard or Weyl modules (see \cite{her:atqchar,her:tqcroot,nak:qvtqchar,nak:tab,nak:tabe} and references therein). The q-characters can be encoded in a ring homomorphism from the Grothendieck ring $R_\xi$ of $\wcal C_\xi$ to the integral group ring $\mathbb Z[\cal P_\xi]$ of the $\ell$-weight lattice. In \cite{nak:qvtqchar}, under the assumption that $\lie g$ is simply laced, Nakajima defined certain polynomials similar to Kazhdan-Lusztig polynomials by studying cohomology of quiver varieties. This lead him to define a function $\chi_{\xi,t}:R_\xi\otimes_\mathbb Z\mathbb Z[t,t^{-1}]\to \mathbb Z[\cal P_\xi]\otimes_\mathbb Z\mathbb Z[t,t^{-1}]$ called the $t$-analogue of the q-character ring homomorphism. It turns out that at $t=1$ this function specializes to the q-character ring homomorphism. Moreover, the definition of $\chi_{\xi,t}$ was axiomatized in a purely combinatorial manner leading to an algorithm for computing the q-characters of the irreducible and standard modules. This algorithm was used in \cite{nak:tab} to give explicit formulas for the $t$-analogues of the q-characters of the standard modules when $\lie g$ is of type $A$ or $D$ and $\xi$ is not a root of unity. The formulas were presented in terms of tableaux and a connection with the theory of crystals was discovered. The algorithm was also used with the help of a supercomputer to compute the $t$-analogues of the q-characters of the fundamental modules when $\lie g$ is of type $E$ in \cite{nak:tabe} (also for generic $\xi$). In \cite{her:atqchar,her:tqcroot}, Hernandez proved a conjecture of Nakajima saying that the existence of the function $\chi_{\xi,t}$ could be established using only its axiomatic description (without the use of geometry). This allowed him to extend the concept of $t$-analogues of q-characters to general $\lie g$. However, due to the lack of a definition of the quiver varieties in general, a proof that the algorithm indeed gives the q-characters of the irreducible modules when $\lie g$ is not simply laced is still missing.

In \cite[Theorem 6.1]{cmq}, Chari and the second author proved that the q-characters of the fundamental representations satisfy a certain invariance property with respect to the braid group action on the $\ell$-weight lattice provided $\lie g$ is of classical type and $\xi$ is not a root of unity.  This result was used in \cite{cm:fund} to obtain closed formulas for the q-characters of fundamental representations of $U_\xi(\tlie g)$ in terms of the braid group action on the $\ell$-weight lattice (the expressions obtained in the earlier literature on $q$-characters did not involve the braid group action). In fact, the theorem gives a recursive method for computing a lower bound for the q-characters of the fundamental representations. In order to show that this is also an upper bound, the authors of \cite{cm:fund} used that the dimensions of the fundamental representations were known. As mentioned above, we prove in this paper that \cite[Theorem 6.1]{cmq} remains valid in the root of unity setting (see Theorem \ref{t:bginv}). Furthermore, by looking at the simplest nontrivial example, we show that the above described procedure may be applied to obtain formulas for the q-characters of fundamental representations in the root of unity setting as well. More precisely, we use Theorem \ref{t:bginv} and the theory of specialization of modules to obtain an explicit formula for the q-character of the fundamental representation corresponding to the adjoint representation when $\lie g$ is of type $D_n$ and $\xi$ is any root of unity (of odd order) in terms of the braid group action on the $\ell$-weight lattice. In particular, we show that this fundamental module is irreducible as \us_\xi g-module iff the order of $\xi$ divides $n-2$ (Nakajima has explained to us that this result can also be deduced from the algorithm of \cite{nak:qvtqchar}). We believe that the same line of reasoning can be performed in general. Since the combinatorics would be rather lengthy, we find it more appropriate to leave this task for a forthcoming publication. We note, however, that there are more $\ell$-weight spaces with multiplicity higher than one in the root of unity setting than in the generic one as it can be seen from the example we considered here. We explicitly describe the multiplicities for this example (see \eqref{e:multdn2}).

The paper is organized as follows. In Section \ref{s:algs}, we fix the basic notation and review the construction and some structural results for the algebra \ust_\xi g. In Section \ref{s:lattices}, we review the definitions of the $\ell$-weight and the $\ell$-root lattices and describe their quotient. Also in Section \ref{s:lattices}, we give the definitions of $\xi$-resonant ordering of dominant $\ell$-weights and of $\xi$-regular dominant $\ell$-weights which are the main combinatorial conditions used in the statements of the main results of Section \ref{s:hlwtensor}. Section \ref{s:fdrs} brings the basic facts of the finite-dimensional representation theory of \ust_\xi g such as the classification of the irreducible modules in terms of dominant $\ell$-weights, the existence of universal finite-dimensional highest-$\ell$-weight modules called the Weyl modules, and a few results related to the q-characters. Section \ref{s:stb} is mostly dedicated to reviewing a few crucial facts used in the proof of our main results such as the basics of the theory of specialization of modules, evaluation representations, and the existence of the Frobenius homomorphism. The main results of the paper are in Sections \ref{s:hlwtensor} (tensor products) and \ref{s:block} (block decomposition and q-characters). Out of the several auxiliary results proved throughout the paper we find worth mentioning Propositions \ref{p:lcharspec} and \ref{p:constspec} as well as Theorem \ref{t:lgenerator}.

\section{The Quantum Loop Algebras at Roots of Unity}\label{s:algs}

\subsection{Notation and Basics on Simple Lie Algebras}\label{ss:basicnot}

Throughout,  $\mathbb C, \mathbb Z,\zpz$ denote the sets of complex  numbers, integers, and non-negative integers. For any integer $m$, $\zs {>m}$ denotes the set of all integers greater than $m$. Given a ring $\mathbb A$, the underlying multiplicative group of units is denoted by $\mathbb A^\times$. The dual of a vector space $V$ is $V^*$. The symbol ``$\cong$'' denotes isomorphisms. Tensor products of vector spaces are always taken over the underlying field unless otherwise indicated.

Let $\lie g$ be a simple Lie algebra of rank $n$ over the complex numbers with a fixed triangular decomposition
$\lie g = \lie n^+\oplus \lie h\oplus \lie n^-$, where $\lie h$ is a Cartan subalgebra and $\lie n^\pm$ are the span of positive (respectively, negative) root vectors. Let $I=\{1,\dots,n\}$ be an indexing set of the vertices of the Dynkin diagram of $\lie g$ and $R^+$ the set of positive roots. The simple roots are denoted by $\alpha_i$, the fundamental weights by $\omega_i$, while $Q,P,Q^+,P^+$ denote the root and weight lattices with corresponding positive cones, respectively.
Let also $h_i$ be the coroot associated to $\alpha_i$ and denote by $x_i^\pm$ a basis element of the root space corresponding to $\pm\alpha_i, i\in I$.  Equip $\lie h^*$ with the partial order $\lambda\le \mu$ iff $\mu-\lambda\in Q^+$.
Denote by $\cal W$ the Weyl group of $\lie g$, by $s_i$ the simple reflections, by $\ell(w)$ the length of $w\in\cal W$,
and let $w_0$ be the longest element of $\cal W$.
Let $C = (c_{ij})_{i,j\in I}$ be the Cartan matrix of $\lie g$, i.e., $c_{ij}=\alpha_j(h_i)$, and let $D = {\rm diag}(d_i:i\in I)$ where the numbers $d_i$ are coprime positive integers such that $DC$ is symmetric.  Recall that the lacing number is $r^\vee=\max\{c_{ij}c_{ji}:i,j\in I, i\ne j\}$. We suppose the nodes of the Dynkin diagram of $\lie g$ are labeled as in Table 1 below and let $I_\bullet$ be the indexing set of the black nodes.

\vspace{16pt}

{\centerline {\bf Table 1}}
\vspace{6pt}
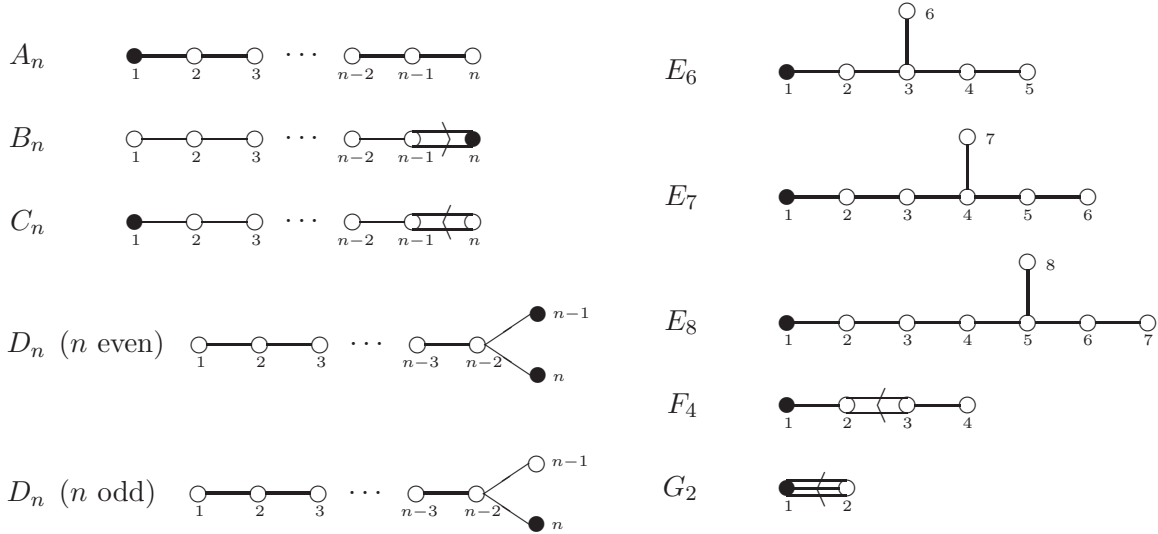
\begin{multicols}{2}
\begin{itemize}
\item[$A_n$]
\setlength{\unitlength}{.2cm}
\begin{picture}(0,0)(-5,-.5)
\linethickness{1pt}
\put(0,0){\circle*{1}}
\put(-0.2,-1.5){$\scriptscriptstyle{1}$}
\put(0.5,0){\line(1,0){3}}
\put(4,0){\circle{1}}
\put(3.8,-1.5){$\scriptscriptstyle{2}$}
\put(4.5,0){\line(1,0){3}}
\put(8,0){\circle{1}}
\put(7.8,-1.5){$\scriptscriptstyle{3}$}
\put(10,0){\ldots}
\put(14.5,0){\circle{1}}
\put(13.5,-1.5){$\scriptscriptstyle{n-2}$}
\put(15,0){\line(1,0){3}}
\put(18.5,0){\circle{1}}
\put(17.5,-1.5){$\scriptscriptstyle{n-1}$}
\put(19,0){\line(1,0){3}}
\put(22.5,0){\circle{1}}
\put(22.2,-1.5){$\scriptscriptstyle{n}$}
\end{picture}
\vspace{15pt}

\item[$B_n$]
\setlength{\unitlength}{.2cm}
\begin{picture}(0,0)(-5,-.5)
\linethickness{0.5pt}
\put(0,0){\circle{1}}
\put(-0.2,-1.5){$\scriptscriptstyle{1}$}
\put(0.5,0){\line(1,0){3}}
\put(4,0){\circle{1}}
\put(3.8,-1.5){$\scriptscriptstyle{2}$}
\put(4.5,0){\line(1,0){3}}
\put(8,0){\circle{1}}
\put(7.8,-1.5){$\scriptscriptstyle{3}$}
\put(10,0){\ldots}
\put(14.5,0){\circle{1}}
\put(13.5,-1.5){$\scriptscriptstyle{n-2}$}
\put(15,0){\line(1,0){3}}
\put(18.5,0){\circle{1}}
\put(17.5,-1.5){$\scriptscriptstyle{n-1}$}
\put(18.5,0.5){\line(1,0){4}}
\put(18.5,-0.5){\line(1,0){4}}
\put(20.3,-0.5){$\big>$}
\put(22.5,0){\circle*{1}}
\put(22.2,-1.5){$\scriptscriptstyle{n}$}
\end{picture}
\vspace{15pt}

\item[$C_n$]
\setlength{\unitlength}{.2cm}
\begin{picture}(0,0)(-5,-.5)
\linethickness{0.5pt}
\put(0,0){\circle*{1}}
\put(-0.2,-1.5){$\scriptscriptstyle{1}$}
\put(0.5,0){\line(1,0){3}}
\put(4,0){\circle{1}}
\put(3.8,-1.5){$\scriptscriptstyle{2}$}
\put(4.5,0){\line(1,0){3}}
\put(8,0){\circle{1}}
\put(7.8,-1.5){$\scriptscriptstyle{3}$}
\put(10,0){\ldots}
\put(14.5,0){\circle{1}}
\put(13.5,-1.5){$\scriptscriptstyle{n-2}$}
\put(15,0){\line(1,0){3}}
\put(18.5,0){\circle{1}}
\put(17.5,-1.5){$\scriptscriptstyle{n-1}$}
\put(18.5,0.5){\line(1,0){4}}
\put(18.5,-0.5){\line(1,0){4}}
\put(20.3,-0.5){$\big<$}
\put(22.5,0){\circle{1}}
\put(22.2,-1.5){$\scriptscriptstyle{n}$}
\end{picture}
\vspace{30pt}

\item[$D_n$] ($n$ even)
\setlength{\unitlength}{.2cm}
\begin{picture}(0,0)(-1,-.5)
\linethickness{1pt}
\put(0,0){\circle{1}}
\put(-0.2,-1.5){$\scriptscriptstyle{1}$}
\put(0.5,0){\line(1,0){3}}
\put(4,0){\circle{1}}
\put(3.8,-1.5){$\scriptscriptstyle{2}$}
\put(4.5,0){\line(1,0){3}}
\put(8,0){\circle{1}}
\put(7.8,-1.5){$\scriptscriptstyle{3}$}
\put(10,0){\ldots}
\put(14.5,0){\circle{1}}
\put(13.5,-1.5){$\scriptscriptstyle{n-3}$}
\put(15,0){\line(1,0){3}}
\put(18.5,0){\circle{1}}
\put(17.7,-1.5){$\scriptscriptstyle{n-2}$}
\put(19,0){\line(3,2){3}}
\put(19,0){\line(3,-2){3}}
\put(22.5,2){\circle*{1}}
\put(22.5,-2){\circle*{1}}
\put(23.5,1.8){$\scriptscriptstyle{n-1}$}
\put(23.5,-2.4){$\scriptscriptstyle{n}$}
\end{picture}
\vspace{40pt}

\item[$D_n$] ($n$ odd)
\setlength{\unitlength}{.2cm}
\begin{picture}(0,0)(-1.5,-.5)
\linethickness{1pt}
\put(0,0){\circle{1}}
\put(-0.2,-1.5){$\scriptscriptstyle{1}$}
\put(0.5,0){\line(1,0){3}}
\put(4,0){\circle{1}}
\put(3.8,-1.5){$\scriptscriptstyle{2}$}
\put(4.5,0){\line(1,0){3}}
\put(8,0){\circle{1}}
\put(7.8,-1.5){$\scriptscriptstyle{3}$}
\put(10,0){\ldots}
\put(14.5,0){\circle{1}}
\put(13.5,-1.5){$\scriptscriptstyle{n-3}$}
\put(15,0){\line(1,0){3}}
\put(18.5,0){\circle{1}}
\put(17.7,-1.5){$\scriptscriptstyle{n-2}$}
\put(19,0){\line(3,2){3}}
\put(19,0){\line(3,-2){3}}
\put(22.5,2){\circle{1}}
\put(22.5,-2){\circle*{1}}
\put(23.5,1.8){$\scriptscriptstyle{n-1}$}
\put(23.5,-2.4){$\scriptscriptstyle{n}$}
\end{picture}

\columnbreak

\vspace*{-35pt}
\item[$E_6$]
\setlength{\unitlength}{.2cm}
\begin{picture}(0,7)(-5,-.5)
\linethickness{1pt}
\put(0,0){\circle*{1}}
\put(-0.2,-1.5){$\scriptscriptstyle{1}$}
\put(0.5,0){\line(1,0){3}}
\put(4,0){\circle{1}}
\put(3.7,-1.5){$\scriptscriptstyle{2}$}
\put(4.5,0){\line(1,0){3}}
\put(8,0){\circle{1}}
\put(7.8,-1.5){$\scriptscriptstyle{3}$}
\put(8.5,0){\line(1,0){3}}
\put(12,0){\circle{1}}
\put(11.8,-1.5){$\scriptscriptstyle{4}$}
\put(12.5,0){\line(1,0){3}}
\put(16,0){\circle{1}}
\put(15.8,-1.5){$\scriptscriptstyle{5}$}
\put(8,0.5){\line(0,1){3}}
\put(8,4){\circle{1}}
\put(9.2,3.5){$\scriptscriptstyle{6}$}
\end{picture}
\vspace{5pt}

\item[$E_7$]
\setlength{\unitlength}{.2cm}
\begin{picture}(0,7)(-5,-.5)
\linethickness{1pt}
\put(0,0){\circle*{1}}
\put(-0.2,-1.5){$\scriptscriptstyle{1}$}
\put(0.5,0){\line(1,0){3}}
\put(4,0){\circle{1}}
\put(3.7,-1.5){$\scriptscriptstyle{2}$}
\put(4.5,0){\line(1,0){3}}
\put(8,0){\circle{1}}
\put(7.7,-1.5){$\scriptscriptstyle{3}$}
\put(8.5,0){\line(1,0){3}}
\put(12,0){\circle{1}}
\put(11.7,-1.5){$\scriptscriptstyle{4}$}
\put(12.5,0){\line(1,0){3}}
\put(16,0){\circle{1}}
\put(15.7,-1.5){$\scriptscriptstyle{5}$}
\put(16.5,0){\line(1,0){3}}
\put(20,0){\circle{1}}
\put(19.7,-1.5){$\scriptscriptstyle{6}$}
\put(12,0.5){\line(0,1){3}}
\put(12,4){\circle{1}}
\put(13.2,3.5){$\scriptscriptstyle{7}$}
\end{picture}
\vspace{5pt}

\item[$E_8$]
\setlength{\unitlength}{.2cm}
\begin{picture}(0,7)(-5,-.5)
\linethickness{1pt}
\put(0,0){\circle*{1}}
\put(-0.2,-1.5){$\scriptscriptstyle{1}$}
\put(0.5,0){\line(1,0){3}}
\put(4,0){\circle{1}}
\put(3.7,-1.5){$\scriptscriptstyle{2}$}
\put(4.5,0){\line(1,0){3}}
\put(8,0){\circle{1}}
\put(7.7,-1.5){$\scriptscriptstyle{3}$}
\put(8.5,0){\line(1,0){3}}
\put(12,0){\circle{1}}
\put(11.7,-1.5){$\scriptscriptstyle{4}$}
\put(12.5,0){\line(1,0){3}}
\put(16,0){\circle{1}}
\put(15.7,-1.5){$\scriptscriptstyle{5}$}
\put(16.5,0){\line(1,0){3}}
\put(20,0){\circle{1}}
\put(19.7,-1.5){$\scriptscriptstyle{6}$}
\put(20.5,0){\line(1,0){3}}
\put(24,0){\circle{1}}
\put(23.7,-1.5){$\scriptscriptstyle{7}$}
\put(16,0.5){\line(0,1){3}}
\put(16,4){\circle{1}}
\put(17.2,3.5){$\scriptscriptstyle{8}$}
\end{picture}
\vspace{15pt}

\item[$F_4$]
\setlength{\unitlength}{.2cm}
\begin{picture}(0,2)(-5,-.5)
\linethickness{0.5pt}
\put(0,0){\circle*{1}}
\put(-0.2,-1.5){$\scriptscriptstyle{1}$}
\put(0.5,0){\line(1,0){3}}
\put(4,0){\circle{1}}
\put(3.7,-1.5){$\scriptscriptstyle{2}$}
\put(4,0.5){\line(1,0){4}}
\put(4,-0.5){\line(1,0){4}}
\put(5.8,-0.5){$\big<$}
\put(8,0){\circle{1}}
\put(7.7,-1.5){$\scriptscriptstyle{3}$}
\put(8.5,0){\line(1,0){3}}
\put(12,0){\circle{1}}
\put(11.7,-1.5){$\scriptscriptstyle{4}$}
\end{picture}
\vspace{15pt}

\item[$G_2$]
\setlength{\unitlength}{.2cm}
\begin{picture}(0,2)(-5,-.5)
\linethickness{0.5pt}
\put(0,0){\circle*{1}}
\put(-0.2,-1.5){$\scriptscriptstyle{1}$}
\put(0,0.5){\line(1,0){4}}
\put(0.5,0){\line(1,0){3}}
\put(0,-0.5){\line(1,0){4}}
\put(1.8,-0.5){$\big<$}
\put(4,0){\circle{1}}
\put(3.7,-1.5){$\scriptscriptstyle{2}$}
\end{picture}
\end{itemize}
\end{multicols}

We shall need the following well-known lemma which can be deduced from the results of \cite{hum:book}.

\begin{lem}\label{l:wgroupstuff}
Let $\lambda\in P^+$.
\begin{enumerate}
\item If $\mu\in P^+$ is such that $\mu\le\lambda$ and $w\in\cal W$, then $w\mu\le\lambda$. Moreover, $w_0\lambda$ is the unique minimal element of the set $P(\lambda):=\{w\mu:w\in\cal W, \mu\in P^+,\mu\le\lambda\}$.
\item If $i\in I$ and $w\in\cal W$ are such that $\ell(s_iw)=\ell(w)+1$, then $w^{-1}\alpha_i\in R^+$. In particular, $w\lambda+\alpha_i\notin P(\lambda)$.\hfill\qedsymbol
\end{enumerate}
\end{lem}

If  $\lie a$ is a Lie algebra over $\mathbb C$, define its loop algebra to be $\tlie a=\lie a\otimes_{\mathbb C}  \mathbb C[t,t^{-1}]$ with bracket given by $[x \otimes t^r,y \otimes t^s]=[x,y] \otimes t^{r+s}$. Clearly $\lie a\otimes 1$ is a subalgebra of $\tlie a$ isomorphic to $\lie a$ and, by abuse of notation, we will continue denoting its elements by $x$ instead of $x\otimes 1$.  We have $\tlie g = \tlie n^-\oplus \tlie h\oplus \tlie n^+$ and $\tlie h$ is an abelian subalgebra. The elements $x_{i}^\pm\otimes t^r$ and $h_i\otimes t^r$ will be denoted by $x_{i,r}^\pm$ and $h_{i,r}$, respectively. Let $\hat C=(c_{ij})_{i,j\in\hat I}, \hat I=I\sqcup\{0\}$, be the extended Cartan matrix of $\lie g$. Then $\tlie g$ is the quotient of the affine Kac-Moody algebra $\hlie g$ associated to $\hat C$ (without derivation) by its one-dimensional center. There exists a unique positive integer $d_0$ such that, if $\hat D = {\rm diag}(d_i,i\in \hat I)$, then $\hat D\hat C$ is symmetric.

Let $U(\lie a)$ denote the universal enveloping algebra of a Lie algebra $\lie a$. Then $U(\lie a)$ is a subalgebra of $U(\tlie a)$. Moreover, if $\lie a$ is a direct sum of two of its subalgebras, say $\lie a=\lie b\oplus\lie c$, then multiplication establishes an isomorphism of vector spaces $U(\lie b)\otimes U(\lie c)\to U(\lie a).$
The assignments $\triangle: \lie a\to U(\lie a)\otimes U(\lie a), x\mapsto x\otimes 1+1\otimes x$, $S:\lie a\to \lie a,
x\mapsto -x$,  and $\epsilon: \lie a\to \mathbb C, x\mapsto 0$, can
be uniquely extended so that $U(\lie a)$ becomes a Hopf algebra with
comultiplication $\triangle$, antipode $S$, and  counit $\epsilon$.
Given a Hopf algebra $H$, we shall denote by $H^0$ the augmentation ideal of  $H$, i.e., the kernel of its counit.

\subsection{Quantum Loop Algebras}
Let $\blb C(q)$ be the ring of rational functions in an indeterminate $q$ and define
\begin{equation*}
[m]_q=\frac{q^m -q^{-m}}{q -q^{-1}},\ \ \ \ [r]_q!
=\prod_{j=1}^r[j]_q,\ \ \ \ \tqbinom{m}{r}_q  = \frac{1}{[r]_q!}\prod_{j=0}^{r-1}[m-j]_q,
\end{equation*}
for $r,m\in\mathbb Z, r\ge 0$. If $m\ge r$, then $\tqbinom{m}{r}_q = \displaystyle{\frac{[m]_q!}{[r]_q![m-r]_q!}}$.
Set $\mathbb A= \bb C [q, q^{-1}]$ and recall  that $[m]_q, [m]_q!,\tqbinom{m}{r}_q \in \mathbb A$. Thus, when $q$ specializes to a non-zero complex number $\zeta$, then $[m]_q, [m]_q!, \tqbinom{m}{r}_q $ specialize to complex numbers which will be denoted by $[m]_\zeta, [m]_\zeta!,\tqbinom{m}{r}_\zeta$, respectively.

Set $q_i=q^{d_i}$ and $[m]_i=[m]_{q_i}$. The quantum affine algebra $U_q(\hlie g)$ is  the algebra over $\mathbb C(q)$ with generators $x_i^\pm, k_i^{\pm 1}, i\in\hat I$, satisfying the following defining relations:
\begin{gather*}
k_ik_i^{-1} = k_i^{-1}k_i =1, \quad  k_ik_j =k_jk_i, \quad k_ix_j^\pm k_i^{-1} = q_i^{\pm c_{ij}}x_j^{\pm},\quad [x_i^+ ,  x_j^-]=\delta_{i,j}  \frac{ k_i-k_i^{-1}}{q_i - q_i^{-1}},\ \ i,j\in\hat I,\\
\sum_{m=0}^{1-c_{ij}} (x_i^{\pm})^{(m)}x_j^{\pm}(x_i^{\pm})^{(1-c_{ij}-m)}=0,\ \ \text{if $i\ne j$}
\end{gather*}
where $(x_i^\pm)^{(m)} = \frac{(x_i^\pm)^m}{[m]_i!}$ for all $m\in\mathbb Z_{\ge 0}, i\in\hat I$. The last set of relations is known under the name of  quantum Serre relations. Let $\theta$ be the highest root of $\lie g$ and write $\theta=\sum_{i\in I} \theta_i\alpha_i$,
$\theta_i\in \bb Z_{\geq 0}$. The quantum loop algebra  $\uqh g$ of $\lie g$ is  the quotient of $U_q(\hlie g)$ by the two-sided ideal generated by $1-k_0\prod_{i\in I}k_i^{\theta_i}$.
It was proved in \cite{be} that \uqt g is isomorphic to the $\mathbb C(q)$-algebra with generators $x_{i,r}^{\pm}, k_i^{\pm 1}, h_{i,s}, i\in I, r,s\in\blb Z, s\ne 0$ satisfying the following defining relations:
\begin{gather*}
k_ik_i^{-1} = k_i^{-1}k_i =1, \ \  k_ik_j =k_jk_i,\ \  k_ih_{j,r}=h_{j,r}k_i ,\; \ \; h_{i,r}h_{j,s}=h_{j,s}h_{i,r},\\
k_ix_{j,r}^\pm k_i^{-1} = q_i^{{}\pm c_{ij}}x_{j,r}^{{}\pm{}},\qquad  [h_{i,r} , x_{j,s}^{{}\pm{}}] =  \pm\frac1r[rc_{ij}]_i x_{j,r+s}^{{}\pm{}},\\
x_{i,r}^{\pm}x_{j,s}^{\pm} -q_i^{\pm c_{ij}}x_{j,s}^{\pm}x_{i,r}^{\pm} =q_i^{\pm c_{ij}}x_{i,r-1}^{\pm}x_{j,s+1}^{\pm} -x_{j,s+1}^{\pm}x_{i,r-1}^{\pm},\\
[x_{i,r}^+ ,  x_{j,s}^-]=\delta_{i,j}  \frac{ \psi_{i,r+s}^+ - \psi_{i,r+s}^-}{q_i - q_i^{-1}},\\
\sum_{\sigma\in S_m}\sum_{k=0}^m(-1)^k
\left[\begin{matrix}m\\k\end{matrix}\right]_{i}
  x_{i, r_{\sigma(1)}}^{{}\pm{}}\ldots x_{i,r_{\sigma(k)}}^{{}\pm{}}
  x_{j,s}^{{}\pm{}} x_{i, r_{\sigma(k+1)}}^{{}\pm{}}\ldots
  x_{i,r_{\sigma(m)}}^{{}\pm{}} =0,\ \ \text{if $i\ne j$},
\end{gather*}

for all sequences of integers $r_1,\ldots, r_m$, where $m
=1-c_{ij}$, $S_m$ is the symmetric group on $m$ letters, $\psi_{i,\mp r}^{{}\pm{}}=0$ if $r>0$, and $\psi_{i,\pm r}^{{}\pm{}}$ for $r\ge 0$ are determined by equating powers of
$u$ in the following equality of formal power series:
$$\Psi_i^\pm(u) := \sum_{r=0}^{\infty}\psi_{i,\pm
r}^{{}\pm{}}u^{r} = k_i^{{}\pm 1}
{\text{exp}}\left(\pm(q_i-q_i^{-1})\sum_{s=1}^{\infty}h_{i,\pm s}
u^{s}\right).$$

This realization is known in the literature as Drinfeld's loop-like realization of  $U_q(\tlie g)$. The above isomorphism sends $x_i^\pm$ to $x_{i,0}^\pm$ for all $i\in I$. Thus, we may denote $x_{i,0}^\pm$ by $x_{i}^\pm$ when convenient and this should not cause confusion. We will mostly use Drinfeld's realization as it is more convenient for studying finite-dimensional representations than the original presentation.

Denote by $\uqh{n^\pm}, \uqh h$  the subalgebras of \uqh g generated by $\{x_{i,r}^\pm\}, \{k_i^{\pm1}, h_{i,s}\}$, respectively. Let  \uq g be  the subalgebra generated by $x_{i,0}^\pm, k_i^{\pm 1}$ and define \uq{n^\pm}, \uq h in the obvious way.  Note that $\psi_{i,0}^{{}\pm{}}=k_i^{{}\pm 1}$ and that \uq g is the quantum group as defined in \cite {lroot}.  Multiplication establishes isomorphisms of $\blb C(q)$-vectors spaces:
$$\uq g \cong \uq{n^-} \otimes \uq h \otimes \uq{n^+} \qquad\text{and}\qquad \uqh g \cong \uqh{n^-} \otimes \uqh h \otimes \uqh{n^+}.$$
We will also consider the subalgebra \uqh{g_i} generated by $k_i^{\pm 1}, h_{i,r}, x^\pm_{i,s}$ for all $r,s\in \blb Z, r\ne 0$. The algebra \uqh{g_i} is isomorphic to $U_{q_i}(\tlie{sl}_2)$. The subalgebra \uq{g_i} is defined similarly.

We shall make use of the following proposition whose proof is straightforward.

\begin{prop}\label{p:ssp}
Let $a\in\mathbb C(q)^\times$. Then, there exists a unique $\mathbb C(q)$-algebra automorphism $\varrho_a$ of $U_q(\tlie g)$ such that $\varrho_a$ is the identity on $U_q(\lie g)$ and $\varrho_a(x_{i,r}^\pm) = a^r x_{i,r}^\pm$  for all $i\in I, r\in \mathbb Z$. Moreover, $\varrho_a(h_{i,r}^\pm) = a^r h_{i,r}^\pm$ for all $i\in I, r\in \mathbb Z$.\hfill\qedsymbol
\end{prop}

\subsection{Restricted Specialization} We now recall the definition of the restricted specializations of \uq g and \uqt g (cf. \cite{cproot, fmroot,lroot}).
For $ r, c \in \blb Z, s,m\in \zp, i\in I$, let $(x_{i,r}^\pm)^{(m)} = \frac{(x_{i,r}^\pm)^m}{[m]_i!}\in\uqt g$ and
\begin{equation}
\left[\begin{matrix} k_i; c\\
s\end{matrix}\right]=\prod_{m=1}^s\frac{k_iq_i^{c+1-m}-k_i^{-1}q_i^{-c-1+m}}
{q_i^m-q_i^{-m}}.
\end{equation}
Denote \tqbinom{k_i;0}{s} by \tqbinom{k_i}{s}. Define elements $\Lambda_{i,r}, i\in I, r\in\blb Z$, of \uqh g by
\begin{equation}
\Lambda_i^\pm(u)=\sum_{r=0}^\infty \Lambda_{i,\pm r} u^{r}=
\exp\left(-\sum_{s=1}^\infty\frac{h_{i,\pm s}}{[s]_{i}}u^{s}\right).
\end{equation}
Note that \begin{equation}
\Psi_i^\pm(u)=k_i^{\pm
1}\frac{\Lambda_i^\pm(uq_i^{\mp 1})}{\Lambda_i^\pm(uq_i^{\pm 1})}
\end{equation}
where the division is that of formal power series in $u$ with coefficients in \uqt g. Notice also that  \uqt h is generated by $k_i^{\pm 1}$ and $\Lambda_{i,r}$, $i\in I, r\in \blb Z$, as a $\mathbb C(q)$-algebra.
Let \uqtres g be the $\mathbb A$-subalgebra of $\uqt g$ generated by the elements $(x_{i,r}^\pm)^{(m)}$ and $k_i$ for all $i\in I, r\in\blb Z, m\in \mathbb Z_{\ge 0}$. Set $\uqtres{n^{\pm}}=\uqtres g\cap \uqt n^\pm$ and $\uqtres h=\uqtres g\cap \uqt h$. It was proved in \cite[Proposition 6.1]{cproot} that multiplication establishes an isomorphism of $\mathbb A$-modules
\begin{equation}
\uqtres g\cong \uqtres{n^-}\otimes\uqtres h \otimes\uqtres{n^+}.
\end{equation}
It also follows that
 \uqtres{n^{\pm}} is the $\mathbb A$-subalgebra of \uqt g generated by $\{(x_{i,r}^\pm)^{(m)}:i\in I, r\in\mathbb Z\}$ and  \uqtres h is the $\mathbb A$-subalgebra of \uqt g generated by  $\{k_i,\tqbinom{k_i}{s}, \Lambda_{i,r}: i\in I, r\in\mathbb Z, s\in\mathbb Z_{\ge 0}\}$.
The subalgebras \uqres g, \uqres{n^\pm}, and \uqres h are defined  similarly and the corresponding statements to the ones above follow as consequences. We note that $\tqbinom{k_i;c}{m}\in \uqres g$ and record the following relations (cf. \cite{l-adv}):
\begin{gather} \label{e:xhcrel}
\bik m^c (x_{i,r}^\pm)^{(p)}= (x_{i,r}^\pm)^{(p)} \bik m^{c\pm pc_{ij}}, \qquad p\in\znn.
\end{gather}

Given $\xi\in\mathbb C^\times$, denote by $\epsilon_\xi$ the evaluation map $\mathbb A \to \blb C$ sending $q$ to $\xi$ and by $\blb C_\xi$ the $\mathbb A$-module obtained by pulling-back $\epsilon_\xi$.
Set
\begin{equation}
\us_\xi a= \blb C_\xi\otimes_{\mathbb A}\uqres a ,
\end{equation}
for $\lie a = \lie g, \lie n^\pm, \lie h, \tlie g, \tlie n^\pm, \tlie h$. The algebra \us_\xi a is called the restricted specialization of $\uq a$ at $q=\xi$. We shall denote an element of the form $1\otimes x\in\us_\xi a$ with $x\in \uqres a$  simply by $x$. If $\xi$ is not a root of unity, the algebra \ust_\xi g is isomorphic to the algebra given by generators and relations analogous to those of \uqt g with $\xi$ in place of $q$ and its representation theory is parallel to that of \uqt g. We shall be particularly interested in the cases $\xi=\zeta$ where $\zeta$ is a nontrivial root of unity and $\xi=1$. In the former case, the algebras \usres a are also known as Lusztig's quantum groups.

Henceforth, assume $l$ is an odd integer $\ge3$ relatively prime to the lacing number of $\lie g$. Let $\zeta\in\blb C$ be a primitive $l$-th root of unity and set $\zeta_i=\zeta^{d_i}$ for $i\in I$. The hypotheses on $l$ imply that $\zeta_i^2$ is a root of $1$ of order $l$ as well. In particular, $l$ is minimal such that $[l]_{\zeta_i}=0$. Let $\mathbb C'$ be the set consisting of $q$ and all nonzero complex numbers whose multiplicative order is either infinity or odd and relatively prime to the lacing number of $\lie g$. Henceforth, for simplicity of notation, the expression ``root of unity'' will pertain only to nontrivial roots of unity and not to $1$. Unless stated otherwise, $\xi$ will always denote an arbitrary element of $\mathbb C'$. We set $l=1$ when $\xi$ is not a root of unity for notational convenience. An element of $\mathbb C'$ which is not a root of unity will be referred to as a generic element (so $1$ and $q$ are generic). The expression ``the order of $\xi$'' will mean the order of the multiplicative subgroup generated by $\xi$. Notice $1$ is the only generic element of finite order.

\begin{rem}
Although several of the results that follow can be proved for roots of unity of even order as well, some of the proofs (and even the precise statements) require many more technicalities. For simplicity and uniformity, we shall keep the above made assumptions on $\xi$.
\end{rem}

The algebra \ust_\xi g is a Hopf algebra with comultiplication given by
\begin{equation}
\Delta (x_i^+) = x_i^+\otimes 1+ k_i\otimes x_i^+, \qquad \Delta (x_i^-) = x_i^-\otimes k_i^{-1}  + 1\otimes x_i^-, \qquad \Delta(k_i) = k_i\otimes k_i
\end{equation}
for all $i\in\hat I$. If $\xi\ne 1$, an expression for the comultiplication $\Delta$ of \ust_\xi g in terms of the generators $x^\pm_{i,r}, h_{i,r}, k_i^{\pm1}$ is not known. The following partial information obtained in \cite{be,bcp,cpmin,dam} suffices for our purposes (see also Lemma 7.5 of \cite{cproot}). Given $m\in\mathbb Z_\ge 0$, let $(X^\pm)^{(m)}$  be the $\mathbb C(\xi)$-span of the elements $(x_{i_1,r_1}^\pm)^{(m_1)}\cdots (x_{i_s,r_s}^\pm)^{(m_s)}$ for all $s\in\mathbb Z_{\ge 0}, i_j\in I, r_j, m_j\in\blb Z, m_j\ge 0, j=1,\dots,s, \sum_{j=1}^sm_j=m$.

\begin{prop}\label{p:comultip}  The restriction of  $\Delta$ to \ust_\xi{g_i} satisfies:
\begin{enumerate}
\item $\Delta(\tqbinom{k_i}{r}) =\sum_{j=0}^r \tqbinom{k_i}{r-j}k_i^{-j}\otimes \tqbinom{k_i}{j}k_i^{r-j}$, for every $r\in\mathbb Z_{\ge 0}$.

\item Modulo $\ust_\xi g X^-\otimes \ust_\xi g(X^+)^{(2)} + \ust_\xi g X^-\otimes \ust_\xi g X^+$,
we have
\begin{align*}
\Delta (x_{i,r}^+)& = x_{i,r}^+\otimes 1+ k_i\otimes x_{i,r}^+ +\sum_{j=1}^r \psi^+_{i,j}\otimes x_{i,r-j}^+ \ \ (r\ge 0),\\
\Delta (x_{i,-r}^+)& =k_i^{-1}\otimes x_{i,-r}^+ + x_{i,-r}^+\otimes 1 +\sum_{j=1}^{r-1} \psi^-_{i, -j}\otimes  x_{i,-r+j}^+ \ \ (r> 0).
\end{align*}

\item  Modulo $\ust_\xi g(X^-)^{(2)}\otimes\ust_\xi gX^+ + \ust_\xi gX^-\otimes \ust_\xi g X^+$,
we have
\begin{align*}
\Delta (x_{i,r}^-)& =x_{i,r}^-\otimes k_i+ 1\otimes x_{i,r}^-+\sum_{j=1}^{r-1} x_{i,r-j}^-\otimes \psi^+_{i, j} \ \ (r>  0),\\
\Delta (x_{i,-r}^-)& = x_{i,-r}^-\otimes k_i^{-1}  + 1\otimes x_{i,-r}^- +\sum_{j=1}^r  x_{i,-r+j}^-\otimes \psi^-_{i, -j} \ \ (r\ge  0).
\end{align*}

\item Modulo $\ust_\xi gX^-\otimes \ust_\xi g X^+$, the following hold for every $r\in\mathbb Z_{>0}$
\begin{gather*}
\Delta(h_{i,r}) =h_{i,r}\otimes 1+1\otimes h_{i,r},\\
\Delta(\Lambda_{i,\pm r}) =\sum_{j=0}^r \Lambda_{i,\pm(r-j)}\otimes \Lambda_{i,\pm j},\qquad\text{and}\qquad
\Delta(\psi_{i,\pm r}^\pm) =\sum_{j=0}^r \psi_{i,\pm(r-j)}^\pm\otimes \psi_{i,\pm j}^\pm
\end{gather*}
\end{enumerate}
\end{prop}

If $V$ is a representation of a Hopf algebra $H$, its dual space $V^*$ is equipped with the structure of a representation by using the antipode. We shall not need the precise expression for the antipode of \uqt g.

\begin{lem}\label{l:hinres}
The elements $\tilde h_{i,s}:=\frac{h_{i,s}}{[s]_i}$ belong to \uqtres g. Moreover, the image of $\tilde h_{i,s}$ in \ustres g is nonzero. In particular, \ustres h is generated by $\{k_i,\tqbinom{k_i}{l}, \tilde h_{i,s}: i\in I, s\in\mathbb Z_{>0}\}$ as a $\mathbb C$-algebra (with identity).
\end{lem}

\begin{proof}
We have
\begin{equation}
-\sum_{s=1}^\infty\tilde h_{i,\pm s}u^{s} = \ln\left(\sum_{r=0}^\infty \Lambda_{i,\pm r} u^{r}\right) = \sum_{t=0}^\infty \frac{(-1)^t}{t+1} \left(\sum_{r=1}^\infty \Lambda_{i,\pm r} u^{r}\right)^{t+1}.
\end{equation}
It follows that
$$\tilde h_{i,\pm s} = \sum_{0\le t\le s} b_{i,s,t} \tilde{\Lambda_{i,t}^{\pm}}$$
with  $b_{i,s,s}=1, b_{i,s,t}\in\mathbb Q$, and $\tilde{\Lambda_{i,t}^{\pm}}=\sum \Lambda_{i,\pm r_0} \cdots \Lambda_{i,\pm r_t}$  where $r_0+\cdots+r_t=s,  1\le r_j\le s$. Clearly the lemma now easily follows.
\end{proof}

Let \ustf_\xi g be the subalgebra of \ust_\xi g generated by \ust_\xi h and $x_{i,r}^\pm, i\in I, r\in\mathbb Z$. Define \usf_\xi g similarly. If $\xi$ is not a root of unity, $\ustf_\xi g=\ust_\xi g$ and $\usf_\xi g=\us_\xi g$, but for $\xi=\zeta$ they are proper subalgebras.

\begin{rem}
The algebras \usfin g and \ustfin g are enlargements (in the Cartan part only) of the so-called small quantum group and small quantum affine algebra (see \cite[\S2.4]{fmroot} for instance).
\end{rem}

\section{Braid Group and $\ell$-Lattices}\label{s:lattices}

\subsection{Quantum factorizations}

Given $a\in\mathbb C(\xi)^\times, r\in\mathbb Z_{\ge 0}$, define
\begin{equation}\label{e:qsf}
f_{a,r}(u) = \prod_{j=0}^{r-1} (1-a\xi^{r-1-2j}u)
\end{equation}
Even though the dependence on $\xi$ is missing in the notation above, this will not create a confusion.
Suppose $\xi$ has infinite order and observe that if $f(u)\in\mathbb C(\xi)[u]$ splits in $\mathbb C(\xi)[u]$ and $f(0)=1$, then there exist unique $m\in\mathbb Z_{\ge 0}$, $a_1,\dots,a_m\in\mathbb C(\xi)^\times,$ and $r_1,\dots,r_m\in\mathbb Z_{\ge 1}$  such that
\begin{equation}\label{e:qfactorp}
f(u) = \prod_{k=1}^m f_{a_k,r_k}(u)\qquad\text{with}\qquad \frac{a_k}{a_j}\ne \xi^{\pm(r_k+r_j-2p)}\quad\text{for}\quad k\ne j,\ 0\le p<\min\{r_k,r_j\}.
\end{equation}

\begin{defn}\label{d:qfactor1}
Assume $\xi$ is of infinite order. The factorization \eqref{e:qfactorp} is called the $\xi$-factorization of $f$ and the factors $f_{a_k,r_k}$ are called the $\xi$-factors of $f$.
\end{defn}

The above factorizations first appeared in \cite{cp:qaa} but the terminology $\xi$-factorization appears to have been first used in \cite{chhe:beyond}. In order to define the concept of $\xi$-factorization when $\xi$ has finite order $l$, consider $\phi_l:\mathbb C(q)[u]\to\mathbb C(q)[u]$, the $\mathbb C(q)$-algebra homomorphism such that $\phi_l(u) = u^l$. Observe that, if $a,b\in\mathbb C(q)$ are such that $b^l=a$, then
\begin{equation}
\phi_l(1-au)=1-au^l = \prod_{j=0}^{l-1} (1-b\xi^ju)
\end{equation}
and for every polynomial $f\in\mathbb C[u]$ with constant term 1 there exist unique polynomials $f',f''$ also with constant term 1 such that $f''$ is not divisible by $1-au^l$ for any $a\in\mathbb C^\times$ and
\begin{equation}\label{e:frobfactorp}
f  =  f'\phi_l(f'').
\end{equation}

Given $f\in\mathbb A(u)$, let $\bar f$ be the element of $\mathbb C(u)$ obtained from $f$ by evaluating $q$ at a given $\xi\in\mathbb C^\times$ (again, the fact that the notation does not make the dependence on $\xi$ explicit will not cause confusion). Then:
\begin{equation}\label{e:skrpoly}
\overline{f_{a,l}(u)}= 1-a^lu^l \qquad\text{for all}\qquad  a\in\mathbb C^\times.
\end{equation}
It is now easy to see that for every $f\in\mathbb C[u]$ satisfying $f(0)=1$, there exist unique $m_0,m_1\in\mathbb Z_{\ge 0}, a_1,\dots,a_{m_0}, b_1,\dots, b_{m_1}\in\mathbb C^\times$, and  $r_1,\dots,r_{m_0}\in\mathbb Z_{>0}$,  such that
\begin{gather}\notag
f(u) = \left(\prod_{k=1}^{m_0} f_{a_k,r_k}(u)\right)\left(\prod_{k=1}^{m_1} (1-b_ku^l)\right)\qquad \text{with}\qquad 0<r_k<l,\\\label{e:sfactorp}
\\ \qquad\text{and}\qquad  \notag \frac{a_k}{a_j}\ne \xi^{\pm(r_k+r_j-2p)} \qquad\text{for}\qquad 0\le p<\min\{r_k,r_j\}.
\end{gather}

\begin{defn}\label{d:qfactor2}
Assume $\xi$ has finite order. The factorization \eqref{e:sfactorp} is called the $\xi$-factorization of $f$. The factors $f_{a_k,r_k}$ are called the quantum $\xi$-factors of $f$ while the factors $1-b_ku^l$ are called the Frobenius $\xi$-factors. We will refer to both kinds of factors simply by the $\xi$-factors of $f$. The notion of multiplicity of a $\xi$-factor is defined in the obvious way.
\end{defn}

\begin{rem}
Notice that if $\xi=1$ the set of quantum factors of the $\xi$-factorization \eqref{e:sfactorp} is empty and \eqref{e:sfactorp} reduces to the usual factorization into a product of polynomials of degree 1. To unify the cases when $\xi$ is of finite and of infinite order, we say that the factors appearing in \eqref{e:qfactorp} are quantum $\xi$-factors and, hence, the set of Frobenius $\xi$-factors is always empty if $\xi$ has infinite order. The terminology ``Frobenius $\xi$-factor'' appears not to have been used before. The motivation for the terminology comes from the second statement of Theorem \ref{t:frobtensor}(b) below.
\end{rem}

\begin{defn}
Let $f\in\mathbb C(\xi)[u]$ be a $\xi$-factorizable polynomial. If $\xi\ne 1$, $f$ is said to be $\xi$-regular if its set of Frobenius $\xi$-factors is empty while $f$ is said to be $1$-regular if its factorization is multiplicity free.
Let $f,g\in\mathbb C[u]$ be $\xi$-factorizable with $\xi$-factors $\{f_{a_k,r_k}, 1-b_ju^l:k=1,\dots,m_0, j=1,\dots,m_1\}$ and $\{f_{a_k',r_k'}, 1-b_j'u^l:k=1,\dots,m_0', j=1,\dots,m_1'\}$, respectively. The ordered pair $(f,g)$ is said to be in $\xi$-resonant order (respectively  weak $\xi$-resonant) if $b_k\ne b_j'$ and
\begin{equation}\label{e:qresonantp}
\frac{a_k}{a_j'}\ne \xi^{-(r_k+r_j'-2p)} \quad\text{for all}\quad k,j\quad\text{and all}\quad 0\le p<r_k \text{ (respectively } 0\le p<\min\{r_k,r_j'\}).
\end{equation}
The polynomials $f$ and $g$ are said to be in general position if both $(f,g)$ and $(g,f)$ are in weak $\xi$-resonant ordering.
An $m$-tuple $(f_1,f_2,\cdots,f_m)$ of polynomials is said to be in (weak) $\xi$-resonant order if $(f_j,f_k)$ is in (weak) $\xi$-resonant order for every $j<k$. The concept of general position for a finite collection of polynomials is defined by requiring that they are pairwise in general position.
\end{defn}

\begin{rem}
Clearly, if $(f,g)$ is in $\xi$-resonant order it is in weak $\xi$-resonant order as well. In general, both $(f,g)$ and $(g,f)$ may not be in weak $\xi$-resonant order. The concept of resonant order was introduced in \cite{vcbraid}, but with different terminology. Namely, the polynomial $g$ was said to be in general position with respect to $f$ if the first condition of \eqref{e:qresonantp} was satisfied. The terminology ``resonant'' was used in \cite{em}. The above given definition of general position coincides with that introduced in \cite{cp:qaa}. Notice that $f$ and $g$ are in general position iff their Frobenius factors are distinct and the set with multiplicities of quantum factors of their product is the union of their sets of quantum factors. Also, if $\xi=1$, the concepts of (weak) resonant ordering and general position are equivalent. Notice that an $m$-tuple of polynomials of the form $(f_{a_1,1},\dots,f_{a_m,1})$ is in $\xi$-resonant order iff $a_j/a_k\ne \xi^{-2}$ for $j<k$.
\end{rem}

The next lemma is easily established.

\begin{lem}\label{l:regular}
Let $f\in\mathbb C(\xi)[u]$ be such that $f = \prod_{j=1}^m f_{a_j,1}$ for some $a_j\in\mathbb C(\xi)^\times, j=1,\dots,m$. Then, there exists $\sigma\in S_m$ such that $(f_{a_{\sigma(1)},1},\cdots, f_{a_{\sigma(m)},1})$ is in $\xi$-resonant order iff $f$ is $\xi$-regular.\hfill\qedsymbol
\end{lem}

\subsection{The $\ell$-weight lattice}\label{ss:lwl}
Given a field $\mathbb F$, consider the multiplicative group $\cal P_\mathbb F$ of $n$-tuples of rational functions $\gb\mu = (\gb\mu_1(u),\cdots, \gb\mu_n(u))$ with values in $\mathbb F$  such that $\gb\mu_i(0)=1$ for all $i\in I$. We shall often think of $\gb\mu_i(u)$ as a formal power series in $u$ with coefficients in $\mathbb F$. Given $a\in\mathbb F^\times$ and $i\in I$, let $\gb\omega_{i,a}$ be defined by
$$(\gb\omega_{i,a})_j(u) = 1-\delta_{i,j}au.$$
Clearly, if $\mathbb F$ is algebraically closed, $\cal P_\mathbb F$ is the free abelian group generated by these elements which are called fundamental $\ell$-weights. It is also convenient to introduce the elements $\gb\omega_{\lambda,a}, \lambda\in P,a\in\mathbb F^\times$, defined by
\begin{equation}
\gb\omega_{\lambda,a} = \prod_{i\in I}(\gb\omega_{i,a})^{\lambda(h_i)}
\end{equation}
and the elements
\begin{equation}\label{e:qsfw}
\gb\omega_{i,a,r} = \prod_{j=0}^{r-1} \gb\omega_{i,a\xi_i^{r-1-2j}},
\end{equation}
$i\in I,a\in\mathbb C(\xi)^\times, r\in\mathbb Z_{\ge 0}$. Even though the dependence on $\xi$ is not made explicit in the notation $\gb\omega_{i,a,r}$, this will not create a confusion. Notice $\gb\omega_{i,a,1}=\gb\omega_{i,a}$.

If $\mathbb F$ is algebraically closed, introduce the group homomorphism (weight map) $\wt:\cal P_\mathbb F \to P$ by setting $\wt(\gb\omega_{i,a})=\omega_i$, where $\omega_i$ is the $i$-th fundamental weight of $\lie g$. Otherwise, let $\mathbb K$ be an algebraically closed extension of $\mathbb F$ and regard $\cal P_\mathbb F$ as a subgroup of $\cal P_\mathbb K$. Then, define the weight map on $\cal P_\mathbb K$ as above and on $\cal P_\mathbb F$ by restriction (this clearly does not depend on the choice of $\mathbb K$).
Define the $\ell$-weight lattice of $U_q(\tlie g)$ to be $\cal P_q:=\cal P_{\mathbb C(q)}$. The submonoid $\cal P_q^+$ of $\cal P_q$ consisting of $n$-tuples of polynomials is called the set of dominant $\ell$-weights of $U_q(\tlie g)$.

Given $\gb\lambda\in\cal P_q^+$ with $\gb\lambda_i(u) = \prod_j (1-a_{i,j}u)$, where $a_{i,j}$ belong to some algebraic closure of $\mathbb C(q)$, let $\gb\lambda^-\in\cal P_q^+$ be defined by $\gb\lambda^-_i(u) = \prod_j (1-a_{i,j}^{-1}u)$. We will also use the notation $\gb\lambda^+ = \gb\lambda$. Two elements $\gb\lambda,\gb\mu$ of $\cal P_q^+$  are said to be relatively prime if $\gb\lambda_i(u)$ is relatively prime to $\gb\mu_j(u)$ in $\mathbb C(q)[u]$ for all $i,j\in I$. Every  $\gb\nu\in\cal P_q$ can be uniquely written in the form
\begin{equation}\label{e:frac}
\gb\nu = \gb\lambda\gb\mu^{-1} \quad\text{with}\quad \gb\lambda,\gb\mu\in\cal P_q^+ \quad\text{relatively prime}.
\end{equation}
Given $\gb\nu = \gb\lambda\gb\mu^{-1}$ as above, define a $\mathbb C(q)$-algebra homomorphism $\gb\Psi_{\gb\nu}:\uqt h\to \mathbb C(q)$ by  setting
\begin{equation}
\gb\Psi_{\gb\nu}(k_i^{\pm 1}) = q_i^{\pm \wt(\gb\nu)(h_i)}, \qquad \sum_{r\ge 0} \gb\Psi_{\gb\nu}(\Lambda_{i,\pm r}) u^r = \frac{(\gb\lambda^{\pm})_i(u)}{(\gb\mu^{\pm})_i(u)}
\end{equation}
where the division is that of formal power series in $u$.

The next proposition is easily checked.
\begin{prop}\label{p:lwl*}
The map $\gb\Psi:\cal P_q\to (U_q(\tlie h))^*$ given by $\gb\nu\mapsto \gb\Psi_{\gb\nu}$ is injective.\hfill\qedsymbol
\end{prop}

Let $\cal P$  be the subgroup of $\cal P_q$ generated by $\gb\omega_{i,a}$ for all $i\in I$ and all $a\in\mathbb C^\times$ (equivalently, $\cal P=\cal P_\mathbb C$). Set also $\cal P^+=\cal P\cap\cal P_q^+$. From now on we will identify $\cal P_q$  with its image in $(U_q(\tlie h))^*$  under $\gb\Psi$. Similarly, given $\xi\in\mathbb C^\times$, $\cal P$ can be identified with a subset of $U_\xi(\tlie h)^*$ by using the same expression for $\gb\Psi_{\gb\nu}(\Lambda_{i,\pm r})$ and
$$\qquad \gb\Psi_{\gb\nu}(k_i^{\pm 1}) = \xi_i^{\pm \wt(\gb\nu)(h_i)} \qquad\text{and}\qquad \gb\Psi_\nu\left(\tqbinom{k_i}{l}\right) = \tqbinom{\wt(\gb\nu)(h_i)}{l}_{\xi_i}.$$
Observe also that every element $\gb\lambda\in \cal P_q$ such that $\gb\lambda_i(u)$ splits in $\mathbb C(q)[u]$ for all $i\in I$ can be uniquely decomposed as
\begin{equation}\label{e:rootfactor}
\gb\lambda = \prod_j \gb\omega_{\lambda_j,a_j} \quad\text{for some}\quad \lambda_j\in P \quad\text{and}\quad a_i\ne a_j.
\end{equation}
For notational convenience, henceforth we set $\cal P_\xi=\cal P$ when $\xi\ne q$.

Recall that we denote by $\phi_l$ the $\mathbb C(q)$-algebra homomorphism $\mathbb C(q)[u]\to\mathbb C(q)[u]$ determined by $u \mapsto u^l$. We also denote by $\phi_l$ the induced group homomorphism $\cal P_q\to\cal P_q$.
Similarly to \eqref{e:frobfactorp}, every $\gb\lambda\in \cal P^+$ admits a unique decomposition of the form
\begin{equation}\label{e:frobfactor}
\gb\lambda(u)  =  \gb\lambda'\phi_l(\gb\lambda'')  \qquad\text{with}\qquad \gb\lambda',\gb\lambda'' \in\cal P^+,
\end{equation}
where $\gb\lambda'_i(u)$ is $\xi_i$-regular for all $i\in I$.  Set $\cal P_l^+ = \{\gb\lambda\in\cal P^+: \gb\lambda=\gb\lambda'\}$.

Denote by $\cal P_\mathbb A^+$ be the subset of $\cal P_q$ consisting of $n$-tuples of polynomials with coefficients in $\mathbb A$ and $\cal P_\mathbb A$ the corresponding subgroup. Let $\cal P_\mathbb A^\times $ be the submonoid of $\cal P_\mathbb A^+$ consisting of $n$-tuples of polynomials with invertible leading coefficients and $\cal P_\mathbb A^s=\{\gb\lambda\in\cal P_q^+: \text{ the roots of }\gb\lambda_i(u) \text{ lie in }\mathbb A^\times\text{ for all }i\in I\}\subseteq \cal P_\mathbb A^\times $. We may also use the terminology $\gb\lambda$ splits in $\mathbb A^\times$ to mean $\gb\lambda\in\cal P_\mathbb A^s$.  Notice that, if $\gb\lambda\in\cal P_\mathbb A^\times $, then $\gb\lambda_i(u)$ splits in $\mathbb A$ for all $i\in I$ iff $\gb\lambda$ splits in $\mathbb A^\times$.
Given $\gb\lambda\in\cal P_\mathbb A$, let \bgb\lambda\ be the element of $\cal P$ obtained from $\gb\lambda$ by evaluating $q$ at $\xi$ (where we again omit the dependence on $\xi$ in the notation).

\subsection{Braid group action on the $\ell$-weight lattice}

The following proposition was proved in \cite{bp,vcbraid}.

\begin{prop}
The following formulas define an action of $\cal B$ on $\cal P_q$:
\begin{gather*}
(T_i\gb\mu)_i(u) = (\gb\mu_i(\xi_i^2u))^{-1} \qquad\text{and}\qquad (T_i\gb\mu)_j(u)  =\gb\mu_j(u)\prod\limits_{k=1}^{|c_{ji}|} \gb\mu_i(q^{d_i+|c_{ji}|+1-2k}u).
\end{gather*}
In particular, $\wt(T_w\gb\mu)=w\wt(\gb\mu)$ for all $w\in \cal W$. Moreover, $T_i\gb\mu\in\cal P_\mathbb A$ for $\gb\mu\in\cal P_\mathbb A$.\hfill\qedsymbol
\end{prop}

By composing the above action with the evaluation map $\epsilon_\xi, \xi\in\mathbb C^\times$, one gets an action of $\cal B$ on $\cal P$.  Hence, if $\gb\mu\in\cal P$, there is an ambiguity in the notation $T_i\gb\mu$. It will always be clear from the context if we mean the ``$q$-action'' or the ``$\xi$-action'' for some $\xi\in\mathbb C^\times$.

\begin{lem}\label{l:Twdominant}
Suppose $i\in I$ and $w\in\cal W$ are such that $\ell(s_iw)=\ell(w)+1$.
Then,  $(T_w\gb\lambda)_i(u)\in \mathbb C(\xi)[u]$ for every $\gb\lambda\in \cal P_\xi^+$.
\end{lem}

\begin{proof}
It follows from \cite[Proposition 3.2]{vcbraid}. Alternatively, see the the proof of Corollary \ref{c:Twdominant} below.
\end{proof}

Recall that $w_0$ defines a Dynkin diagram automorphism such that $w_0\cdot i =j$ iff $w_0\omega_i=-\omega_j$ for  $i,j\in I$. Given $\gb\lambda\in\cal P_\xi^+$, set
\begin{equation}
\gb\lambda^* = (T_{w_0}\gb\lambda)^{-1}.
\end{equation}
A straightforward but tedious computation working with a reduced expression for $w_0$ shows that
\begin{equation}
(\gb\lambda^*)_i(u) = \gb\lambda_{w_0\cdot i}(\xi^{r^\vee h^\vee}u)
\end{equation}
where $h^\vee$ is the dual Coxeter number of $\lie g$ and  $r^\vee$ is the lacing number of $\lie g$.

The next definition is crucial to the results of \S\ref{s:hlwtensor}.

\begin{defn}\label{d:resonant}
Let $\xi\ne 1$. A pair $(\gb\lambda,\gb\mu)$ of dominant $\ell$-weights is said to be in (weak) $\xi$-resonant order if there exists a reduced expression for $w_0$, say $s_{i_N}\cdots s_{i_1}$, such that
\begin{enumerate}
\item $(\gb\lambda_i(u),\gb\mu_i(u))$ is in (weak) $\xi_{i}$-resonant order for every $i\in I$,
\item $((T_{i_{j-1}}\cdots T_{i_1}\gb\lambda')_{i_j}(u),\gb\mu'_{i_j}(u))$ is in (weak) $\xi_{i_j}$-resonant order for every $j=1,\dots, N$.
\item $(T_{i_{j-1}}\cdots T_{i_1}\gb\lambda')_{i_j}(u)$ is $\xi_{i_j}$-regular for all $j=1,\dots, N$.
\end{enumerate}
For $\xi=1$, the pair $(\gb\lambda,\gb\mu)$ is said to be in (weak) $\xi$-resonant order if the pair of polynomials $(\prod_{i\in I}\gb\lambda_i(u)),\prod_{i\in I}\gb\mu_i(u))$ is in (weak) $\xi$-resonant order.
The pair $(\gb\lambda,\gb\mu)$ is said to be in general position if both $(\gb\lambda,\gb\mu)$ and $(\gb\mu,\gb\lambda)$ are in weak $\xi$-resonant ordering. An $m$-tuple of elements of $\cal P_\xi^+$ is said to be in (weak) $\xi$-resonant order if any  ordered pair extracted from the $m$-tuple is in (weak) $\xi$-resonant order. We say that $\gb\lambda=\prod_{j=1}^m\gb\omega_{i_j,a_j}\in\cal P_\xi^+$ is $\xi$-regular  if there exists $\sigma\in S_m$ such that $(\gb\omega_{i_{\sigma(1)},a_{\sigma(1)}}, \dots, \gb\omega_{i_{\sigma(m)},a_{\sigma(m)}})$ is in $\xi$-resonant order.
\end{defn}

\begin{rem}
If $\xi$ is generic, condition (a) above follows from condition (b) (see the proof of \cite[Theorem 4.4]{vcbraid}) while condition (c) is vacuous. Therefore, if $\xi$ has infinite order, the above definition of $\xi$-resonant ordering coincides with the condition used in the main theorem of \cite{vcbraid}. If $\lie g=\lie{sl}_2$, regardless of the value of $\xi$, condition (b) is equivalent to (a) while (c) is vacuous. For $\xi=1$, the concepts of general position and (weak) resonant ordering are equivalent. The definition of $\xi$-regularity given above is motivated by Lemma \ref{l:regular}.
\end{rem}

\begin{lem}\label{l:cyclic}
Let $k\in\mathbb Z_{\ge 1}, i_j\in I, m_j\in\mathbb Z_{\ge 1}, a_j\in\mathbb C(q)^\times$ for $j=1,\dots, k$. Then:
\begin{enumerate}
\item If $\frac{a_{j}}{a_{s}}\notin q^{\mathbb Z_{> 0}}$ for $j>s$, then $(\gb\omega_{i_1,a_1,m_1},\cdots, \gb\omega_{i_k,a_k,m_k})$ is in $q$-resonant order.
\item There exists $\sigma\in S_k$ (independent of $i_1,\cdots, i_k$) such that $(\gb\omega_{i_1,a_{\sigma(1)},m_1},\cdots, \gb\omega_{i_k,a_{\sigma(k)},m_k})$ is in $q$-resonant order.
\end{enumerate}
\end{lem}

\begin{proof}
Part (a) is proved by a direct computation (see \cite[\S6]{vcbraid}). Part (b) easily follows from (a).
\end{proof}

We now recall the definition of certain elements of $\cal P_\xi$ which are the $\ell$-analogue of simple roots. These element were defined originally in \cite{fr,fmroot} using the quantum Cartan matrix (in a similar fashion that the simple roots are expressed in terms of fundamental weights using the coefficients of the Cartan matrix). Later, an alternate definition was given in \cite{cmq} making use of the braid group action on $\cal P_\xi$ instead of the quantum Cartan matrix. We follow the point of view of \cite{cmq}. Given $i\in I, a\in\blb C(\xi)^\times$, define the simple $\ell$-root $\gb\alpha_{i,a}$ by
\begin{equation}
\gb\alpha_{i,a} := \gb\omega_{i,a}(T_i\gb\omega_{i,a})^{-1}  = (\gb\omega_{i,a\xi_i,2})^{-1}\prod_{j\ne i} \gb\omega_{j,a\xi_i,-c_{j,i}}.
\end{equation}
The subgroup of $\cal P_\xi$ generated by the simple $\ell$-roots $\gb\alpha_{i,a}, i\in I, a\in\mathbb C(\xi)^\times$, is called the $\ell$-root lattice of $U_\xi(\tlie g)$ and will be denoted by $\cal Q_\xi$. Let also $\cal Q_\xi^+$ be the submonoid generated by the simple $\ell$-roots. Quite clearly $\wt(\gb\alpha_{i,a})=\alpha_i$. Define a partial order on $\cal P_\xi$ by
$$\gb\mu\le\gb\lambda \qquad{if}\qquad \gb\lambda\gb\mu^{-1}\in\cal Q_\xi^+.$$
One easily checks the following proposition (cf. \cite[\S3.3]{cmq}).
\begin{prop}
The action of the braid group on $\cal P_\xi$ preserves  $\cal Q_\xi$.
\hfill\qedsymbol
\end{prop}

\begin{rem}
The elements $\gb\alpha_{i,a}$ are denoted by $A_{i,aq_i}$ in \cite{fmroot,fr} and in several other papers that follow the notation originally developed in \cite{fr}. Also, the elements $\gb\omega_{i,a}$ are denoted there by $Y_{i,a}$. We find the notation $\gb\omega_{i,a}$ and $\gb\alpha_{i,a}$ introduced in \cite{cmq} more suggestive. For instance, $\wt(\gb\omega_{i,a})=\omega_i$ and $\wt(\gb\alpha_{i,a})=\alpha_i$.
\end{rem}

\subsection{The group of elliptic characters}\label{ss:ec} This subsection will be used only in \S\ref{ss:blocks}. Set
\begin{equation}
\widetilde\Gamma_\xi = \cal P_\xi/\cal Q_\xi.
\end{equation}
The elements of the group $\widetilde\Gamma_\xi$ are called elliptic characters. The motivation for this name will be explained in the remark after Proposition \ref{p:Gammarel} below. We denote by $\gamma_{_\xi}$ the canonical projection $\cal P_\xi\to \widetilde\Gamma_\xi$.

Define subsets $\cal T_k$ of $\mathbb Z_{\ge 0}, k=1,2,3$, as follows.
\begin{alignat*}{3}
&\cal T_1=\{0,2,\cdots,2n\},\quad \cal T_2=\cal T_3=\emptyset, &\quad&\text{ if }\lie g\text{ is of type } A_n,\\
&\cal T_1=\{0,4n-2\},\quad \cal T_2=\cal T_3=\emptyset, &\quad&\text{ if }\lie g\text{ is of type } B_n,\\
&\cal T_1=\{0,2n+2\},\quad \cal T_2=\cal T_3=\emptyset, &\quad&\text{ if }\lie g\text{ is of type } C_n,\\
&\cal T_1=\{0,2,2n-2,2n\},\quad \cal T_2=\cal T_3=\emptyset, &\quad&\text{ if }\lie g\text{ is of type } D_n,\\
&\cal T_1=\{0,8,16\},\quad \cal T_2=\{0,2,4,12,14,16\}, \quad \cal T_3=\emptyset, &\quad&\text{ if }\lie g\text{ is of type } E_6,\\
&\cal T_1=\{0,18\},\quad \cal T_2=\{0,2,12,14,24,26\}, \quad \cal T_3=\emptyset, &\quad&\text{ if }\lie g\text{ is of type } E_7,\\
&\cal T_1=\{0,30\},\quad \cal T_2=\{0,20,40\}, \quad \cal T_3=\{0,12,24,36,48\}, &\quad&\text{ if }\lie g\text{ is of type } E_8,\\
&\cal T_1=\{0,18\},\quad \cal T_2=\{0,12,24\}, \quad \cal T_3=\emptyset, &\quad&\text{ if }\lie g\text{ is of type } F_4,\\
&\cal T_1=\{0,12\},\quad \cal T_2=\{0,8,16\}, \quad \cal T_3=\emptyset, &\quad&\text{ if }\lie g\text{ is of type } G_2.\\
\end{alignat*}
Given $k\in\{1,2,3\}$ and $a\in\mathbb C(\xi)^\times$, define
\begin{equation}
\gb\tau_{\xi,k,a} = \prod_{r\in\cal T_k} \gb\omega_{i,a\xi^r}
\end{equation}
if $\lie g$ is not of type $D_n$ with $n$ even and where $\{i\}=I_\bullet$ (see Table 1). For $\lie g$ of type $D_n$ with $n$ even define
\begin{gather}\notag
\gb\tau_{\xi,1,a} = \gb\omega_{n-1,a}\gb\omega_{n-1,a\xi^{2n-2}}, \qquad \gb\tau_{\xi,2,a} = \gb\omega_{n,a}\gb\omega_{n,a\xi^{2n-2}},\\\text{and}\\\notag
\gb\tau_{\xi,3,a} = \gb\omega_{n-1,a}\gb\omega_{n-1,a\xi^2}\gb\omega_{n,a\xi^{2n-2}}\gb\omega_{n,a\xi^{2n}}.
\end{gather}

\begin{rem}
We warn the reader that there were a couple of typos in \cite[\S A.2]{cmq} in the definition of the sets corresponding to our $\cal T_k$. They have been corrected above.
\end{rem}

\begin{prop}\label{p:Gammarel}
Assume $\lie g$ is not of type $D_n$ with $n$ even. Then, the group $\widetilde\Gamma_\xi$ is isomorphic to the (additive) abelian group with generators $\{\chi_a:a\in\mathbb C(\xi)^\times\}$ and relations
$$\sum_{j\in \cal T_k} \chi_{a\xi^j}=0,$$
for all $a\in\mathbb C(\xi)^\times$ and $k=1,2,3$.  If $\lie g$ is of type $D_n$ with $n$ even, then $\widetilde\Gamma_\xi$ is isomorphic to the (additive) abelian group with generators $\{\chi^\pm_a:a\in\mathbb C(\xi)^\times\}$ and relations
\begin{equation*}
\chi^\pm_a+\chi^\pm_{a\xi^{2n-2}}=0,\ \ \chi^-_a+\chi^-_{a\xi^2}+\chi^+_{a\xi^{2n-2}}+\chi^+_{a\xi^{2n}}=0,
\end{equation*}
for all $a\in\mathbb C(\xi)^\times$.
\end{prop}

\begin{proof}
The proposition was proved in \cite{cmq} for $\xi=q$ and the proof is analogous for the other cases. The isomorphism is given by sending the image of $\gb\omega_{i,a}$ in $\widetilde\Gamma_\xi$ to the generator $\chi_a$ for $i\in I_\bullet$ (in the case $I_\bullet$ has two elements, the image of $\gb\omega_{n-1,a}$ is sent to $\chi^-_a$ while the image of $\gb\omega_{n,a}$ is sent to $\chi^+_a$).
\end{proof}

\begin{rem}
Observe that the relations above imply $\chi_a=\chi_{a\xi^{2r^\vee h^\vee}}$ for every $a\in\mathbb C(\xi)^\times$ (and similarly for $\chi^\pm_a$). This is the motivation for the terminology elliptic characters. Namely, if $\xi\in\mathbb C^\times$ has infinite order, the parameterizing set for the distinct generators of the form $\chi_a$ is the elliptic curve $\mathbb C^\times/ \xi^{2r^\vee h^\vee\mathbb Z}$. For $\xi=1$, the elliptic behavior degenerates and the elements of $\widetilde\Gamma_1$ are usually called spectral characters. Notice that $\widetilde\Gamma_1$ is isomorphic to the group of functions with finite support from $\mathbb C^\times$ to $P/Q$ (cf. \cite{cms}). If $\xi=\zeta$, the curve parameterizing the distinct generators $\chi_a$ depends on the greatest common divisor of $2r^\vee h^\vee$ and the order $l$ of $\zeta$. For instance, let $\lie g=\lie{sl}_2$ so that $l$ is relatively prime to $2r^\vee h^\vee=4$. One can check that we have $\chi_a=\chi_{a\zeta^m}$ for any $m\in\mathbb Z$ and, hence, the curve is $\mathbb C^\times/\zeta^{\mathbb Z}$. Moreover, we also have $2\chi_a=0$ for every $a\in\mathbb C^\times$. In particular, $\widetilde\Gamma_\zeta$ is a quotient of $\widetilde\Gamma_1$ in this case. On the other hand, if $\lie g=\lie{sl}_3$ and $l=3=r^\vee h^\vee$, the curve becomes $\mathbb C^\times$ as in the case $\xi=1$.
\end{rem}

We record the following lemma proved ``in between the lines'' in \cite{cmq,em} for $\xi$ of infinite order. For $\xi\in\{\zeta,1\}$ the proof is analogous. Given $i\in I$, let $\cal P_{\xi,i}$ be the subgroup of $\cal P_\xi$ generated by $\gb\omega_{i,a}$ with $a\in\mathbb C(\xi)^\times$ and $\cal P_{\xi,i}^+=\cal P_\xi^+\cap\cal P_{\xi,i}$. Similarly one defines $\cal P_{\xi,J}$ and $\cal P_{\xi,J}^+$ for any $J\subseteq I$.

\begin{lem}\label{l:samebseq}
If $\gb\lambda,\gb\mu\in\cal P_{\xi,I_\bullet}^+$ are such that $\gamma_{_\xi}(\gb\lambda)=\gamma_{_\xi}(\gb\mu)$, then there exists a sequence $\gb\omega_1,\dots\gb\omega_m\in\cal P_{\xi,I_\bullet}^+$ such that $\gb\omega_1=\gb\lambda, \gb\omega_m=\gb\mu$ and, for every $j=1,\dots,m-1$,
$$\frac{\gb\omega_j}{\gb\omega_{j+1}} = (\gb\tau_{\xi,k_j,a_j})^{\epsilon_j},$$
for some $k_j\in\{1,2,3\}, \epsilon_j\in\{\pm1\}$, and $a_j\in\mathbb C(\xi)^\times$. Moreover, if $\gb\lambda,\gb\mu\in\cal P_\mathbb A^s$, then $a_j$ can be taken in $\mathbb A^\times$ for all $j=1,\dots,m-1$.\hfill\qedsymbol
\end{lem}

\section{The Irreducible and Highest-Weight Representations}\label{s:fdrs}

\subsection{Weight Modules}
We now review the finite-dimensional representation theory of \us_\xi g. All the proofs can be found in \cite{anpoke,cpbook,lroot,l} (see also \cite{jm:weyl}).

Let $V$ be a \us_\xi g-module. Given $\mu\in P$ define
\begin{equation}
V_\mu = \left\{v\in V: k_iv=\xi_i^{\mu(h_i)}v \text{ and } \tqbinom{k_i}{l}v=\tqbinom{\mu(h_i)}{l}_{\xi_i}v\text{ for all } i\in I\right\}.
\end{equation}
If $\xi$ has infinite order, the second condition in the definition of $V_\mu$ is redundant. The space $V_\mu$ is said to be the weight space of $V$ of weight $\mu$ and the nonzero vectors of $\mu$ are called weight vectors of weight $\mu$. The module $V$ is said to be a weight module (of type 1) if
$$V=\opl_{\mu\in P}^{} V_\mu.$$
It is well-known that every finite-dimensional \us_\xi g-module  is isomorphic to the tensor product of a weight module of type 1 with a one-dimensional representation. We shall only consider representations of type 1 (without loss of generality) with finite-dimensional weight spaces and will refer to them simply as weight modules. The abelian tensor category of finite-dimensional \us_\xi g-weight modules will be denoted by $\cal C_\xi$.

Observe that the commutation relations \eqref{e:xhcrel} imply that
\begin{equation}\label{e:shiftws}
(x_i^\pm)^{(k)}V_\mu \subseteq V_{\mu \pm k\alpha_i} \qquad\text{for all}\qquad k\in\mathbb Z_{\ge 0},\ \ \mu\in P.
\end{equation}

\begin{defn}
A \us_\xi g-module $V$ is said to be integrable if, given $v\in V$ and $i\in I$, there exists $m\in\mathbb Z_{\ge 0}$ such that $(x_i^{\pm})^{(k)}v=0$ for all $k>m$.
\end{defn}

Quite clearly, every finite-dimensional \us_\xi g-module is integrable. If $V\in\cal C_\xi$, define the character of $V$ to be the element $\ch(V)$ of the integral group ring $\mathbb Z[P]$ given by
\begin{equation}
\ch(V) = \sum_{\mu\in P} \dim(V_\mu)e^\mu.
\end{equation}
Here, $e^\mu,\mu\in P$, denote the basis elements of $\mathbb Z[P]$ so that $e^\mu e^\nu=e^{\mu+\nu}$. The action of $\cal W$ on $P$ can be naturally extended to an action of $\cal W$ on $\mathbb Z[P]$.

\begin{prop}\label{p:Winv}
Let $V$ be an integrable \us_\xi g-module, $\mu\in P$, and $w\in\cal W$. Then, $V_\mu$ and $V_{w\mu}$ are isomorphic as vector spaces. In particular, the characters of objects in $\cal C_\xi$ are $\cal W$-invariant.
\hfill\qedsymbol
\end{prop}

If $v\in V$ is a weight vector such that $\us_\xi{n^+}^0v = 0$, then $v$ is called a highest-weight vector. If $V$ is generated by a highest-weight vector of weight $\lambda$, then $V$ is said to be a highest-weight module of highest weight $\lambda$. The concept of lowest-weight module is defined similarly by replacing $\us_\xi{n^+}^0$ with $\us_\xi{n^-}^0$.
The next proposition is standard.

\begin{prop}
Let $V$ be a highest-weight module. Then, $V$ has a unique maximal proper submodule and, hence, a unique irreducible quotient. In particular, $V$ is indecomposable.\hfill\qedsymbol
\end{prop}

\begin{defn}
Given $\lambda\in P^+$,  let $W_\xi(\lambda)$ be the universal \us_\xi g-module with respect to the following property: $W_\xi(\lambda)$ is generated by a highest-weight vector $v$ of weight $\lambda$ satisfying $(x_i^-)^{(m)}v = 0$ for all  $i\in I$ and $m>\lambda(h_i)$. The module $W_\xi(\lambda)$ is called the Weyl module of highest weight $\lambda$. The unique irreducible quotient of $W_\xi(\lambda)$ is denoted by $V_\xi(\lambda)$.
\end{defn}

\begin{thm}
For every $\lambda\in P^+$, $W_\xi(\lambda)$ is an integrable \us_\xi g-module. Furthermore, it is the universal finite-dimensional highest-weight $\us_\xi g$-module of highest weight $\lambda$.\hfill\qedsymbol
\end{thm}

It is quite easy to see that any finite-dimensional irreducible \us_\xi g-module (of type 1) is highest-weight. Moreover, since every element of $P$ is conjugate to a dominant weight, it follows from Proposition \ref{p:Winv} that the highest weight must be in $P^+$. This proves:

\begin{cor}
Every simple object from $\cal C_\xi$ is isomorphic to $V_\xi(\lambda)$ for some $\lambda\in P^+$.\hfill\qedsymbol
\end{cor}

The proof of the next theorem is standard.

\begin{thm}\label{t:chconst}
The irreducible constituents of $V\in\cal C_\xi$ (counted with multiplicities) are completely determined by $\ch(V)$.\hfill\qedsymbol
\end{thm}

The above results with $\xi=1$ recover the basic results on the finite-dimensional representation theory of $\lie g$. We shall use the notation $V(\lambda)$ for the $\lie g$-module corresponding to $V_1(\lambda)$.

\begin{prop}\label{p:ufinir}
If $\lambda\in P^+_l= \{\lambda\in P^+:\lambda(h_i)<l\text{ for all }i\in I\}$, then $V_\zeta(\lambda)$ is irreducible as a \usfin g-module.\hfill\qedsymbol
\end{prop}

\subsection{$\ell$-Weight Modules}
This section is based on \cite{cpbook,cproot,fmroot,jm:weyl} and we refer to these papers for the omitted details. A few results intersect with those of \cite{beka} where the finite-dimensional representation theory of the De Concini-Kac specialization of \uqt g was studied (see also \cite{dchere}).

Let $V$ be a \ust_\xi g-module. A nonzero vector $v\in V$ is said to be an $\ell$-weight vector if there exists $\gb\lambda\in\cal P_\xi$ and $k\in\mathbb Z_{>0}$ such that $(\eta-\gb\Psi_\gb\lambda(\eta))^kv=0$ for all $\eta\in \ust_\xi h$. In that case, $\gb\lambda$ is said to be the $\ell$-weight of $v$. Let $V_\gb\lambda$ denote the subspace spanned by all $\ell$-weight vectors of $\ell$-weight $\gb\lambda$. $V$ is said to be an $\ell$-weight module if $V=\opl_{\gb\mu\in\cal P_\xi}^{} V_\gb\mu$.  Denote by $\wcal C_\xi$  the category of all finite-dimensional \ust_\xi g-$\ell$-weight modules.
Quite clearly $\wcal C_\xi$ is an abelian category.
Observe that if $V$ is an $\ell$-weight module, then $V$ is also a \us_\xi g-weight module and
\begin{equation}\label{e:lws}
V_\lambda = \opl_{\gb\lambda:\wt(\gb\lambda)=\lambda}^{} V_\gb\lambda.
\end{equation}
The following proposition was proved in \cite{fr,fmroot}.

\begin{prop}
Let $\xi\in\mathbb C'\backslash\{q\}$. If $V$ is a finite-dimensional \ust_\xi g-module which is a \us_\xi g-weight-module, then $V$ is an $\ell$-weight module.\hfill\qedsymbol
\end{prop}

\begin{rem}
Note that not every finite-dimensional \uqt g-module lying in $\cal C_q$ is an $\ell$-weight module since $\mathbb C(q)$ is not algebraically closed. Instead, they are quasi-$\ell$-weight modules in a sense similar to that developed in  \cite{jm:hlanac} for hyper loop algebras. We shall not need this concept here.
\end{rem}

An $\ell$-weight vector $v$ of $\ell$-weight $\gb\lambda$ is said to be a highest-$\ell$-weight vector if $\eta v=\gb\Psi_\gb\lambda(\eta)v$ for every $\eta\in U_\xi(\tlie h)$ and $\ust_\xi{n^+}^0v=0$. A module $V$ is said to be a highest-$\ell$-weight module if it is generated by a highest-$\ell$-weight vector.
The notion of a lowest-$\ell$-weight module is defined similarly. If $V$ is a highest-$\ell$-weight module of highest $\ell$-weight $\gb\lambda$, then \eqref{e:xhcrel} implies
\begin{equation}\label{e:hw}
\dim(V_{\wt(\gb\lambda)}) = 1\qquad\text{and}\qquad V_\mu\ne 0 \ \text{ only if }\ \mu\le\wt(\gb\lambda).
\end{equation}
The next proposition is easily established using \eqref{e:hw}.

\begin{prop}
If $V$ is a highest-$\ell$-weight module, then it has a unique proper submodule and, hence, a unique irreducible quotient. In particular, $V$ is indecomposable.\hfill\qedsymbol
\end{prop}

\begin{prop}\label{p:drinpoly}
Let $V$ be a finite-dimensional highest-$\ell$-weight module of \ust_\xi g  of highest $\ell$-weight $\gb\lambda\in\cal P_\xi$, $v$ a highest-$\ell$-weight vector of $V$, and $\lambda=\wt(\gb\lambda)$. Then, $\gb\lambda \in \cal P_\xi^+$ and $(x^-_{i,r})^{(m)}v=0$ for all $i\in I, r\in\z, m\in\zs {>\lambda(h_i)}$.\hfill\qedsymbol
\end{prop}

\begin{defn}
Let $\gb\lambda\in \cal P_\xi^+$ and $\lambda=\wt(\gb\lambda)$. The Weyl module $W_\xi(\gb\lambda)$ of highest $\ell$-weight $\gb\lambda$ is the universal \ust_\xi g-module with respect to the following property: $W_\xi(\gb\lambda)$ is generated by a highest-$\ell$-weight vector $v$ of $\ell$-weight $\gb\lambda$ satisfying  $(x_{i,r}^-)^{(m)}v=0$ for all $m>\lambda(h_i)$. Denote by $V_\xi(\gb\lambda)$ the irreducible quotient of $W_\xi(\gb\lambda)$.
\end{defn}

\begin{thm}\label{t:weylm}
For every $\gb\lambda\in\cal P_\xi^+$, $W_\xi(\gb\lambda)$ is nonzero and finite-dimensional. In particular, $W_\xi(\gb\lambda)$ is the universal finite-dimensional highest-$\ell$-weight \ust_\xi g-module of highest $\ell$-weight $\gb\lambda$.\hfill\qedsymbol
\end{thm}

\begin{cor}
If $\xi\in\mathbb C'\backslash\{q\}$ and $V\in\wcal C_\xi$ is simple, then $V\cong V_\xi(\gb\lambda)$ for some $\gb\lambda\in \cal P^+$.\hfill\qedsymbol
\end{cor}

\begin{rem}
Note that not every irreducible finite-dimensional \uqt g-module is highest $\ell$-weight since $\mathbb C(q)$ is not algebraically closed. The classification of the irreducible finite-dimensional \uqt g-modules can be obtained  similarly to what was done in \cite{jm:hlanac} for the hyper loop algebras (cf. \cite{jm:weyl}).
\end{rem}

The following result about dual representations was proved in \cite{fmq,fmroot} for $\xi\ne 1$. For $\xi=1$ it follows easily from Theorem \ref{t:tev} below.

\begin{prop}\label{p:uqtdual}
For every $\gb\lambda\in\cal P_\xi^+$, $V_\xi(\gb\lambda)$ is a lowest-$\ell$-weight module with lowest $\ell$-weight $(\gb\lambda^*)^{-1}$. In particular, $V_\xi(\gb\lambda)^* \cong V_\xi(\gb\lambda^*)$.\hfill\qedsymbol
\end{prop}

The above results for \ust_1 g recover the basic results on the finite-dimensional representation theory of $\tlie g$. We may use the notation $V(\gb\lambda)$ and $W(\gb\lambda)$ for the $\tlie g$-modules corresponding to $V_1(\gb\lambda)$ and $W_1(\gb\lambda)$, respectively.


\begin{prop}\label{p:lufinir}
If $\gb\lambda\in\cal P^+_l$, then $V_\zeta(\gb\lambda)$ is irreducible as a \ustfin g-module.\hfill\qedsymbol
\end{prop}

We will also need the following theorem which follows from the results of \cite{becnak:2sc,kas:cb,kas:cbmod}.

\begin{thm}\label{t:lwmdim}
Let $\gb\lambda\in\cal P_\xi^+, \lambda=\sum_{i\in I} s_i\omega_i=\wt(\gb\lambda)$. Then,
$\ch(W_\xi(\gb\lambda)) = \prod_{i\in I} s_i\ch(V_q(\gb\omega_{i,1}))$.\hfill\qedsymbol
\end{thm}

\subsection{The q-characters}

Let $\mathbb Z[\cal P_\xi]$ be the integral group ring over $\cal P_\xi$. If $V$ is an $\ell$-weight module, the q-character of $V$ is defined to be the following element of $\mathbb Z[\cal P_\xi]$:
\begin{equation}
\chl(V) = \sum_{\gb\mu\in\cal P_\xi} \dim(V_\gb\mu)\gb\mu.
\end{equation}
Consider the following partial order on $\mathbb Z[\cal P_\xi]$. Given $\eta\in\mathbb Z[\cal P_\xi]$, write $\eta=\sum_{\gb\mu\in\cal P_\xi} \eta(\gb\mu)\gb\mu$ and define
\begin{equation*}
\eta \le \eta' \qquad\text{iff}\qquad \eta(\gb\mu)\le \eta'(\gb\mu) \qquad\text{for all}\qquad \gb\mu\in\cal P_\xi.
\end{equation*}
We shall say $\eta$ is an upper (lower) bound for $\chl(V)$ provided $\chl(V)\le \eta$ ($\eta\le\chl(V)$). The evaluation map $\epsilon_\xi$ induces a ring homomorphism $\mathbb Z[\cal P_q]\to \mathbb Z[\cal P]$ also denoted by $\epsilon_\xi$.

\begin{rem}
The original definition of q-characters, given by Frenkel and Reshetikhin in \cite{fr}, was motivated by the study of deformed $\cal W$-algebras and used the concept of transfer matrices which involves the $R$-matrix of $U_q(\hlie g)$. Because of the quantum nature of the original definition, the name q-character was chosen. It was proved in \cite[Proposition 2.4]{fmq} that the definition given above coincides with the definition of \cite{fr}.  Following the convention of adding the prefix ``$\ell$-'' to the name of classical concepts for naming their loop analogues, the name $\ell$-characters was used instead of q-characters in \cite{cmq,cm:fund,jm:hla,jm:hlam,jm:hlanac} (notice that quantum groups are not present at all in \cite{jm:hla,jm:hlam,jm:hlanac}). Since the term q-character has become traditional and it is more widespread in the literature, we follow here the original terminology of \cite{fr} for naming this concept. We remark that we use the term q-character as the name of the concept only and the prefix ``q-'' should not be confused with the letter chosen to denote the quantum parameter.
\end{rem}

The proof of the following theorem is analogous to that of Theorem \ref{t:chconst}.

\begin{thm}\label{t:chlconst}
Let $V\in\wcal C_\xi$. The irreducible constituents (counted with multiplicities) of $V$ are completely determined by $\chl(V)$.\hfill\qedsymbol
\end{thm}

It turns out that the q-characters of objects from $\wcal C_\xi$ are not invariant with respect to the braid group action on $\cal P_\xi$ in general. In fact, the theory of q-characters is much more intricate than that of characters and it is not yet completely understood unless $\xi=1$.  Even the proof of the following basic theorem is involved when $\xi\ne 1$ \cite{fmq,fmroot} (see \cite{cmq} for a different approach when $\lie g$ is of classical type). For $\xi=1$ this theorem follows easily from Theorem \ref{t:tev} below (see also \cite{cms}).

\begin{thm}\label{t:cone}
Let $V$ be a finite-dimensional highest-$\ell$-weight module of highest $\ell$-weight $\gb\lambda\in\cal P_\xi^+$. If $\gb\lambda_i(u)$ splits in $\mathbb C(\xi)[u]$ for every $i\in I$, then $V$ is an $\ell$-weight module and $V_\gb\mu\ne 0$ only if $\gb\mu\le \gb\lambda$. More precisely, if $\gb\lambda = \prod_{j=1}^m \gb\omega_{\lambda_j,a_j}$ for some $\lambda_j\in P^+$ and $a_j\in\mathbb C(\xi)^\times$ and if $V_\gb\mu\ne 0$, then $\gb\mu = \gb\lambda\left(\prod_{k=1}^{m'}\gb\alpha_{i_k,a_{j_k}\xi^{b_k}}\right)^{-1}$ for some $i_k\in I, j_k\in\{1,\dots,m\}$, and $0\le b_k< r^\vee h^\vee$. In particular, $\gb\mu = \prod_{k=1}^{m''} \gb\omega_{i'_k,a_{j'_k}\xi^{b'_k}}$ for some $i'_k\in I, j'_k\in\{1,\dots,m\}$, and $0\le b'_k\le r^\vee h^\vee$. \hfill\qedsymbol
\end{thm}

\begin{rem}
If $\gb\lambda_i(u)$ does not split for some $i\in I$ (in particular $\xi=q$), $V$ is not an $\ell$-weight module, but rather a quasi-$\ell$-weight module, as mentioned previously. An analogue of the above theorem in this case can be established by following the ideas of  \cite{jm:hlanac}. Observe that it follows that $W_q(\gb\lambda)\in\wcal C_q$ iff the roots of $\gb\lambda_i(u)$ are in $\mathbb C(q)$ for all $i\in I$.
\end{rem}

It follows from \cite[Corollary 2]{fr} (see also \cite[Lemma 5.4]{cmq}) that
\begin{equation}\label{e:multchl}
\chl(V\otimes W) = \chl(V)\chl(W)
\end{equation}
for every $V,W\in\wcal C_\xi$. In particular, this and Theorem \ref{t:chlconst} imply:

\begin{cor}\label{c:constensor}
Let $V_1,\dots, V_m\in\wcal C_\xi$ and $\sigma\in S_m$. The irreducible constituents of $V_{\sigma(1)}\otimes\cdots\otimes V_{\sigma(m)}$ (counted with multiplicities) do not depend on $\sigma$.\hfill\qedsymbol
\end{cor}

\begin{rem}
While $\cal C_\xi$ is a braided tensor category for any value of $\xi$, $\wcal C_\xi$ is not braided when $\xi\ne 1$.
\end{rem}

\section{Specialization of Modules, Evaluation Modules, and Frobenius Pullbacks}\label{s:stb}

\subsection{Specialization of modules}\label{ss:lattices}
In this subsection, let $\xi\in\mathbb C'\backslash\{q\}$. Recall that an $\mathbb A$-lattice of a $\mathbb C(q)$-vector space $V$ is the $\mathbb A$-span of a basis of $V$. If $V$ is a \uqt g-module (resp. \uq g-module) and $L$ is an $\mathbb A$-lattice of $V$, we will say that $L$ is a \uqtres g-admissible lattice (resp. \uqres g-admissible lattice) if it is a \uqtres g-submodule (resp. \uqres g-submodule) of $V$. Given a \uqtres g-admissible lattice (resp. \uqres g-admissible lattice) $L$ of a module $V$, define
\begin{equation}\label{e:clm}
\bar L =L\otimes_\mathbb A \mathbb C_\xi.
\end{equation}
Then $\bar L$ is a \ust_\xi g-module (resp. \us_\xi g-module) and $\dim_\mathbb C(\bar L)=\dim_{\mathbb C(q)}(V)$.
Given $v\in L$, we shall denote an element of the form $v\otimes 1\in\bar L$ simply by $v$.
The proof of the following proposition can be found in \cite[Proposition 4.2]{l-adv}.

\begin{prop}\label{p:uqlat}\hfill\\ \vspace{-15pt}
\begin{enumerate}
\item Let $V\in\cal C_q$ and $L$ be a \uqres g-admissible lattice of $V$. Then, $L=\opl_{\mu\in P}^{} L\cap V_\mu$. In particular, $\ch(\bar L)=\ch(V)$.
\item Let $\lambda\in P^+$, $v$ a highest-weight vector of $V_q(\lambda)$, and $L=\uqres gv$. Then, $L$ is a \uqres g-admissible lattice of $V$.\hfill\qedsymbol
\end{enumerate}
\end{prop}

Recall the definitions of $\cal P_\mathbb A^\times $ and $\cal P_\mathbb A^s$ given in \S\ref{ss:lwl}. Recall also that, for $\gb\mu\in\cal P_\mathbb A$, $\bgb\mu$ denotes the element of $\cal P$ obtained from $\gb\mu$ by applying the evaluation map $\epsilon_\xi$. The following theorem was proved in \cite{cpweyl} for simply laced $\lie g$ and in \cite{cha:fer} for $\lie g$ of lacing number $r^\vee=2$. The $G_2$-case presents a few extra subtleties and the proof will appear in the survey \cite{jm:weyl}.

\begin{thm}\label{t:uqtlat}
Let $V$ be a \uqt g-highest-$\ell$-weight module with highest-$\ell$-weight $\gb\lambda\in \cal P_\mathbb A^\times $, $v$  a highest-$\ell$-weight vector of $V$, and $L=U_\mathbb A(\tlie g)v$. Then, $L$ is a $U_\mathbb A(\tlie g)$-admissible lattice of $V, \bar L$ is a quotient of $W_\xi(\bgb\lambda)$, and $\ch(\bar L)=\ch(V)$.\hfill\qedsymbol
\end{thm}

From now on, given a highest-$\ell$-weight module $V$ as in the above proposition, we will denote by $\overline{V}$ the \ust_\xi g-module $\bar L$ as constructed there.

\begin{prop}\label{p:lcharspec}
Let $L$ be a \uqtres g-admissible lattice of $V\in\wcal C_q$. Then, $\chl(\overline V) = \epsilon_\xi(\chl(V))$.
\end{prop}

\begin{proof}
Let $B=\{v_1,\dots, v_m\}$ be an $\mathbb A$-basis of $L$ and, hence, also a $\mathbb C(q)$-basis of $V$. Given $i\in I$ and $r\in\mathbb Z$, let $B_{i,r}$ be the matrix of the action of $\Lambda_{i,r}$ on $V$ with respect to $B$. Then, the entries $b_{s,t}$ of $B_{i,r}$ are in $\mathbb A$ and the characteristic polynomial $p(u)$ of the action of $\Lambda_{i,r}$ on $V$ is $p(u)=\det(u{\rm Id}-B_{i,r})\in\mathbb A[u]$. It follows that the entries of the matrix of the action of $\Lambda_{i,r}$ on $\overline L$ with respect to the basis $\overline B$ are $\epsilon_\xi(b_{s,t})$ and, hence, the characteristic polynomial is $\overline{p(u)}$. The proposition follows.
\end{proof}

\begin{rem}
A version of Proposition \ref{p:lcharspec} was stated in \cite[Theorem 3.2]{fmroot} for $V$ irreducible and $L$ as in Theorem \ref{t:uqtlat}. However, the statement of \cite[Theorem 3.2]{fmroot} is missing the hypothesis $\gb\lambda\in\cal P_\mathbb A^s$. This hypothesis is necessary since otherwise $V\notin\wcal C_q$ and hence $V \supsetneqq\opl_{\gb\mu}^{} V_\gb\mu$. In Example \ref{ex:uqtlat} below we shall see that the ``loop'' analogue of the first statement of Proposition \ref{p:uqlat}(a) does not hold. Namely, if $L$ is an \uqtres g-admissible lattice of $V\in\wcal C_q$, it may not be true that $L=\opl_{\gb\mu\in\cal P_q}^{} L\cap V_\gb\mu$ even if $V$ is a highest-$\ell$-weight module with highest $\ell$-weight $\gb\lambda\in\cal P_\mathbb A^s$ and $L$ is as in Theorem \ref{t:uqtlat}.
\end{rem}

Given $V\in\wcal C_\xi$, set
\begin{equation}
\wtl(V) = \{\gb\mu\in\cal P_\xi: V_\gb\mu\ne 0\}.
\end{equation}

\begin{cor}\label{c:nodom}
Let $\gb\lambda\in\cal P_\mathbb A^\times $ and suppose $\epsilon_\xi(\wtl(V_q(\gb\lambda)))\cap\cal P_\xi^+=\{\bgb\lambda\}$. Then, $\overline{V_q(\gb\lambda)}\cong V_\xi(\bgb\lambda)$.
\end{cor}

\begin{proof}
Since $V_\xi(\bgb\lambda)$ is a quotient of $\overline{V_q(\gb\lambda)}$, it suffices to show that the latter is irreducible. If this were not the case, $\overline{V_q(\gb\lambda)}$ would have to contain an irreducible submodule and, hence, $\epsilon_\xi(\wtl(V_q(\gb\lambda)))$ would have to  contain an element of $\cal P_\xi^+$ other than $\bgb\lambda$.
\end{proof}

\begin{rem}
Notice that Theorem \ref{t:cone} implies that, if $\wtl(V_q(\gb\lambda))\cap\cal P_q^+=\{\gb\lambda\}$ and either $\xi$ is not a root of unity or $\xi=\zeta$ and $l>2r^\vee h^\vee$, then the hypothesis of the above corollary is satisfied.
\end{rem}

Let $\lie g=\lie{sl}_2$, $I=\{i\}$, $v$ be a highest-weight vector of $V_q(\lambda), \lambda\in P^+$, and $L=\uqres gv$. Then it is easy to see that
\begin{equation}\label{e:sl2lat}
\{(x_i^-)^{(k)}v: 0\le k\le\lambda(h_i)\} \quad\text{is an $\mathbb A$-basis of $L$ and }\quad (x_i^+)^{(\lambda(h_i))}(x_i^-)^{(\lambda(h_i))}v =v.
\end{equation}
In particular, the image of this set in $\bar L$ is a $\mathbb C$-basis of $\bar L$.
Now let $\lie g$ be general again, $v$ and $L$ be as in Theorem \ref{t:uqtlat}, and $\lambda=\wt(\gb\lambda)$. Fix a reduced expression $s_{i_N}\cdots s_{i_1}$ for $w\in\cal W$, and  consider
\begin{equation}
m_j = (s_{i_{j-1}}\cdots s_{i_1}\lambda)(h_{i_j}).
\end{equation}
Proceeding inductively on $j$, it follows from Lemma \ref{l:wgroupstuff}(b) and \eqref{e:sl2lat} that $m_j\in\mathbb Z_{\ge 0}$,
\begin{equation}\label{e:lwlat}
(x_{i_1}^+)^{(m_1)}\cdots (x_{i_j}^+)^{(m_j)}(x_{i_j}^-)^{(m_j)}\cdots (x_{i_1}^-)^{(m_1)}v=v,
\end{equation}
and $\{(x_{i_j}^-)^{(m_j)}\cdots (x_{i_1}^-)^{(m_1)}v\}$ is an $\mathbb A$-basis of $L\cap V_{s_{i_j}\cdots s_{i_1}\lambda}$. Set
\begin{equation}\label{e:v_w}
v_w = (x_{i_N}^-)^{(m_N)}\cdots (x_{i_1}^-)^{(m_1)}v.
\end{equation}

\begin{prop}\label{p:Twv}
Let $V, v, w, \lambda,\gb\lambda, m_j$ be as above. Then, $V_{T_w\gb\lambda}=V_{w\lambda}=\mathbb Cv_w$.
\end{prop}

\begin{proof}
As mentioned above, $\{v_w\}$ is a basis for $V_{w\lambda}$ and we have the last equality. Since $\wt(T_w\gb\lambda)=w\lambda$, it follows from \eqref{e:lws} that it remains to show that $v_w\in V_{T_w\gb\lambda}$. This was done in \cite[Proposition 4.1]{vcbraid}.
\end{proof}

\begin{cor}\label{c:Twv}
Let $V$ be a finite-dimensional highest-$\ell$-weight \ust_\xi g-module of highest $\ell$-weight $\gb\lambda\in\cal P_\xi^+$. Then, $V_{T_w\gb\lambda}=V_{w\wt(\gb\lambda)}$ for all $w\in\cal W$.
\end{cor}

\begin{proof}
For $\xi=q$ this is Proposition \ref{p:Twv} again. Otherwise, since $\cal P_\xi^+\subseteq\cal P_\mathbb A^s$, the result is immediate from Propositions \ref{p:Twv} and \ref{p:lcharspec}.
\end{proof}

\begin{cor}\label{c:Twdominant}
Let $\xi\in\mathbb C'$ and $\gb\lambda\in\cal P_\xi^+$ be such that $\gb\lambda_i(u)$ splits in $\mathbb C(q)[u]$ for all $i\in I$ and suppose that $i_0\in I$ and $w\in\cal W$ are such that $\ell(s_{i_0}w)=\ell(w)+1$. Then, $T_w\gb\lambda\le\gb\lambda$ and $(T_w\gb\lambda)_{i_0}(u)\in\mathbb C(\xi)[u]$ splits. Moreover, the roots of $T_w\gb\lambda_{i_0}(u)$ form a subset of $\cup_k a_k\xi^\mathbb Z$ where $a_k$ runs through the set of roots of $\gb\lambda_i(u), i\in I$.
\end{cor}

\begin{proof}
The proof goes inductively on $\ell(w)$. If $\ell(w)=0$ there is nothing to be done. Otherwise, let $j\in I$ and $w'\in\cal W$ be such that $w=s_jw'$ and $\ell(w')=\ell(w)-1$. By the induction hypothesis, $T_{w'}\gb\lambda\le\gb\lambda$. In particular, by the very definition of $\gb\alpha_{i,a}$,  $T_{w'}\gb\lambda = \gb\mu\gb\nu^{-1}$ for unique relatively prime $\gb\mu,\gb\nu\in\cal P_\xi^+$ such that $\gb\mu_i(u),\gb\nu_i(u)$ split in $\mathbb C(\xi)[u]$ for all $i\in I$. Also by the induction hypothesis, $(T_{w'}\gb\lambda)_j(u)\in\mathbb C(\xi)[u]$, i.e., $\gb\nu_j(u)=1$ and the roots of $\gb\mu_j(u)$ lie in the set $\cup_k a_k\xi^\mathbb Z$.

Let $V=V_\xi(\gb\lambda)$ and $v_w$ be defined as in \eqref{e:v_w}. It follows from Lemma \ref{l:wgroupstuff}(b) that $\ust_\xi{n^+_j}^0v_{w'}=0$ and, hence, that $\ust_\xi{g_j}v_{w'}$ is finite-dimensional $\ust_\xi{g_j}$-highest-$\ell$-weight module of highest-$\ell$-weight $(T_{w'}\gb\lambda)_j(u)$.
It now follows from Theorem \ref{t:cone} that $T_w\gb\lambda\le T_{w'}\gb\lambda$. Moreover, by Lemma \ref{l:wgroupstuff}(b) again, $\ust_\xi{n^+_{i_0}}^0v_{w}=0$ and $\ust_\xi{g_{i_0}}v_w$ is a finite-dimensional $\ust_\xi{g_{i_0}}$-highest-$\ell$-weight module of highest-$\ell$-weight $(T_{w}\gb\lambda_{i_0})(u)$. In particular, $(T_{w}\gb\lambda)_{i_0}(u)$ is polynomial which splits in $\mathbb C(\xi)[u]$. Using Theorem \ref{t:cone} once more the last statement follows.
\end{proof}

\begin{rem}
For $\xi$ of infinite order, the above corollary was first stated in \cite{cmq}. As observed in \cite{chhe:beyond}, the proof of \cite{cmq} was incomplete. The approach we used above is a representation theoretic alternative to the one presented in \cite[\S2.10]{chhe:beyond} which relies instead purely on the Weyl group action on $P$ and the braid group action on $\cal P_\xi$.
\end{rem}

\subsection{Evaluation modules}\label{ss:evm}
Given a $\lie g$-module $V$, let $V(a)$ be the $\tlie g$-module obtained by pulling-back the evaluation map ${\rm ev}_a:\tlie g\to\lie g, a\in\mathbb C^\times$. Such modules are called evaluation modules. If $V=V(\lambda)$ we use the notation $V(\lambda,a)$ for the corresponding evaluation module. The next theorem can be found in \cite{cpweyl}.

\begin{thm}\label{t:tev}
Let $\gb\lambda\in\cal P^+$.
\begin{enumerate}
\item If $\gb\lambda = \gb\omega_{\lambda,a}$ for some $\lambda\in P^+$ and some $a\in\mathbb C^\times$, then $V(\gb\lambda)\cong V(\lambda,a)$.
\item If $\gb\lambda = \prod_j \gb\omega_{\lambda_j,a_j}$ with $\lambda_j\in P^+$ and $a_i\ne a_j$ when $i\ne j$, then $V(\gb\lambda)\cong \otm_j^{} V(\lambda_j,a_j)$ and $W(\gb\lambda)\cong\otm_j^{} W(\gb\omega_{\lambda_j,a_j})$.\hfill\qedsymbol
\end{enumerate}
\end{thm}

In the remainder of this subsection we assume $\lie g=\lie{sl}_{n+1}$ and $\xi\in\mathbb C'\backslash\{q\}$. We need an extended version of the quantum group \us_\xi g. Namely, consider the $\mathbb C(q)$-algebra $U_q'(\lie g)$ given by generators $x_i^\pm, k_\mu^{\pm 1}$ with $i\in I, \mu\in P$, and the following defining relations:
\begin{gather*}
k_\mu k_\mu^{-1} = k_\mu^{-1}k_\mu =1, \quad k_\mu k_\nu = k_{\mu+\nu}, \quad k_\mu x_{j}^\pm k_\mu^{-1} = q^{\pm\mu(h_j)}x_{j}^{\pm},\quad  [x_{i}^+ , x_{j}^-]=\delta_{i,j} \frac{k_{\alpha_i} - k_{\alpha_i}^{-1}}{q - q^{-1}}
\end{gather*}
and the quantum Serre's relations. There is an obvious monomorphism of algebras $U_q(\lie g)\to U_q'(\lie g)$ such that $k_i\mapsto k_{\alpha_i}$. A representation $V$ of $U_q'(\lie g)$ is said to be of type 1 if the generators $k_\nu, \nu\in P$, act diagonally with eigenvalues of the form $q^{(\nu,\mu)}$ for some $\mu\in P$ where $(\cdot,\cdot)$ is the bilinear form such that $(\alpha_i,\alpha_j)=c_{ij}$. It is not difficult to see that restriction establishes an equivalence of categories from that of type 1 finite-dimensional $U_q'(\lie g)$-modules to $\cal C_q$. From now on we shall identify these two categories using this equivalence. One can construct the algebras $U_\mathbb A'(\lie g)$ and $U_\xi'(\lie g)$ similarly as \uqres g and \us_\xi g.
The next proposition was proved in \cite[\S2]{jim:qan} and \cite[Proposition 3.4]{cp:small}.

\begin{thm}\label{t:qev}
Suppose $\lie g$ is of type $A_n$. There exists an algebra homomorphism ${\rm ev}: \uqt g\to U_q'(\lie g)$ such that, if $\lambda\in P^+$ and $V$ is the pull-back of $V_q(\lambda)$ by ${\rm ev}$, then there exists $m(\lambda)\in\mathbb Z$ such that $V$ is isomorphic to $V_q(\gb\lambda)$ where
$$\gb\lambda = \prod_{i\in I}\gb\omega_{i,a_i,\lambda(h_i)} \qquad\text{with}\qquad a_1 = q^{m(\lambda)} \quad\text{and}\quad \frac{a_{i+1}}{a_i} = q^{\lambda(h_i)+\lambda(h_{i+1})+1} \quad\text{for}\quad i<n.$$
Moreover, the image of  \uqtres g by ${\rm ev}$ is contained in $U_\mathbb A'(\lie g)$. \hfill\qedsymbol
\end{thm}

Given $a\in\mathbb C(q)^\times$, let
\begin{equation}
{\rm ev}_a = {\rm ev}\circ\varrho_a
\end{equation}
where $\varrho_a$ is defined in Proposition \ref{p:ssp}. Denote by $V_q(\lambda,a)$ the pull-back of $V_q(\lambda)$ by the evaluation map ${\rm ev}_a$. It is easy to see that $V_q(\lambda,a)\cong V_q(\gb\lambda) $ where
\begin{equation}\label{e:evdrin}
\gb\lambda = \prod_{i\in I}\gb\omega_{i,a_i,\lambda(h_i)} \qquad\text{with}\qquad a_1 = aq^{m(\lambda)} \quad\text{and}\quad \frac{a_{i+1}}{a_i} = q^{\lambda(h_i)+\lambda(h_{i+1})+1} \quad\text{for}\quad i<n.
\end{equation}

\begin{rem}
It turns out that if $\lie g\ne\lie{sl}_2$ and $\xi\ne 1$, the simple objects from $\wcal C_\xi$ cannot be realized as tensor products of evaluation modules in general (cf. Theorem \ref{t:tev}). In fact, it is known (see \cite{cha:fer} for instance) that there exists $i\in I$ and $m\in\mathbb Z_{\ge 0}$ such that the action of $U_q(\lie g)$ on $V_q(m\omega_i)$ cannot be extended to one of $U_q(\tlie g)$. This is one of the main reasons behind the fact that the theory of q-characters in the quantum context is much more intricate than in the classical one. Namely, for $\xi=1$ the q-character of the irreducible representations can be easily computed using Theorem \ref{t:tev} and the character of the simple $\lie g$-modules. In the quantum case, unless $\lie g$ is of type $A$, evaluation modules do not exist and, hence, Theorem \ref{t:tev} makes no sense. For type $A$ one could expect that Theorem \ref{t:tev} would still hold and, hence, the q-character theory of simple modules would be reduced to that of the evaluation ones. However, as already mentioned, it is known that this is true only for $\lie g=\lie{sl}_2$. We remark also that there exists a weaker quantum analogue of Theorem \ref{t:tev} for any $\lie g$. Namely, every irreducible module is a tensor product of the so-called prime modules (every evaluation module is prime). The study of q-characters of prime modules is far more complicated than those of evaluation modules. In fact, to the best of our knowledge, the classification of the Drinfeld polynomials corresponding to prime modules is not known. For more details on prime modules we refer the reader to \cite{her:simple} and the references therein.
\end{rem}

Given $a\in\mathbb C^\times$ and $\lambda\in P^+$, let $\mu=-w_0\lambda$ and $b=aq^{-n-1}=aq^{-r^\vee h^\vee}$. Then, it is not difficult to see using Proposition \ref{p:uqtdual} that $V_q(\mu,b)^*$ is isomorphic to $V_q(\gb\lambda)$ where $\gb\lambda$ is such that
\begin{equation}\label{e:evdrin*}
\gb\lambda = \prod_{i\in I}\gb\omega_{i,a_i,\lambda(h_i)} \qquad\text{with}\qquad a_n = aq^{m(\mu)} \quad\text{and}\quad \frac{a_{i+1}}{a_i} = q^{-(\lambda(h_i)+\lambda(h_{i+1})+1)} \quad\text{for}\quad i<n.
\end{equation}

Denote also by ${\rm ev}$ the induced map ${\rm ev}: \ust_\xi g\to U_\xi'(\lie g)$. Similarly, given $a\in\mathbb C^\times$, one defines the evaluation map ${\rm ev}_a:\ust_\xi g\to U_\xi'(\lie g)$.  Denote by $V_\xi(\lambda,a)$ the pull-back of $V_\xi(\lambda)$ by ${\rm ev}_a$ and, similarly, let $W_\xi(\lambda,a)$ be the  pull-back of $W_\xi(\lambda)$.

\begin{prop}\label{p:sevrep}
Let $\lambda\in P^+$ and $a\in\mathbb C^\times$. Then $V_\xi(\lambda,a)$ is the irreducible quotient of $\overline{V_q(\lambda,a)}$. In particular, $V_\xi(\lambda,a)\cong V_\xi(\bgb\lambda)$ where $\gb\lambda$ is as in \eqref{e:evdrin}.
\end{prop}

\begin{proof}
Let $v$ and $L$ be as in Theorem \ref{t:uqtlat}. It follows that $L=\uqres gv$ and, therefore, the irreducible quotient of $\overline{V_q(\lambda,a)}$ is isomorphic to $V_\xi(\lambda)$ as a \us_\xi g-module.
The result now follows easily.
\end{proof}

The following Proposition was proved in \cite{abenak:root,cpmin} for $\xi\ne 1$ (for $\xi=1$ it follows from Theorem \ref{t:tev}).

\begin{prop}
Let $\xi\in\mathbb C'\backslash\{q\}$,  $\gb\lambda\in\cal P_\xi^+$ and $\lambda=\wt(\gb\lambda)$. Then, $V_\xi(\gb\lambda)\cong V_\xi(\lambda)$ as \us_\xi g-module iff $\gb\lambda$ is as in \eqref{e:evdrin} or \eqref{e:evdrin*} with $\xi$ in place of $q$.\hfill\qedsymbol
\end{prop}

We now give explicit formulas for the action of the elements of \ust_\xi g on evaluation modules in the case $\lie g=\lie{sl}_2$. In this case one can normalize ${\rm ev}$ so that $m(\lambda)=0$ for all $\lambda\in P^+$ and, hence, $V_q(\lambda,a)\cong V_q(\gb\omega_{i,a,\lambda(h_i)})$ where $I=\{i\}$. Given $\lambda\in P^+$, let $v_0^\lambda$ be a highest-weight vector of $W_\xi(\lambda)$ and set $v_k^\lambda = (x_i^-)^{(k)}v_{k-1}^\lambda$, for $0<k\le\lambda(h_i)$. Then, $\{v_k^\lambda:0\le k\le\lambda(h_i)\}$ is a basis of $W_\xi(\lambda)$. For convenience, set $v_{-1}^\lambda=v_{\lambda(h_i)+1}^\lambda=0$. The next lemma was proved in \cite[\S4.2]{cp:qaa}.

\begin{lem}\label{l:sl2actev}
The following holds in $W_\xi(\lambda,a)$:
\begin{gather*}
x_{i,r}^+v_k^\lambda = (a\xi^{\lambda(h_i)-2k+2})^r[m-k+1]_\xi\ v_{k-1}^\lambda,\qquad
x_{i,r}^-v_k^\lambda = (a\xi^{\lambda(h_i)-2k})^r[k+1]_\xi\ v_{k+1}^\lambda,\\\text{and}\\
\psi^+_{i,s} v_0^\lambda = (\xi-\xi^{-1})(a\xi^{\lambda(h_i)})^s[\lambda(h_i)]_\xi\ v_0^\lambda,
\end{gather*}
for every $r,s\in\mathbb Z, s>0$.\hfill\qedsymbol
\end{lem}

\subsection{The Frobenius homomorphism}\label{ss:frob}

Recall the Frobenius homomorphism defined in \cite{l} (see also \cite{cproot,fmroot}).

\begin{thm}\label{t:frob}
There exists a Hopf algebra homomorphism $\widetilde{\rm Fr}_\zeta:\ustres g\to U(\tlie g)$ such that
\begin{equation*}
\widetilde{\rm Fr}_\zeta(k_i)=1, \qquad
\widetilde{\rm Fr}_\zeta((x_{i,r}^\pm)^{(k)}) =
\begin{cases}
\frac{(x_{i,r}^\pm)^{k/l}}{(k/l)!}, &\text{ if } l \text{ divides } k,\\
0, &\text{ otherwise.}
\end{cases}
\end{equation*}
In particular,
\begin{equation}
\widetilde{\rm Fr}_\zeta(\Lambda_{i,r}) =
\begin{cases}
\Lambda_{i,r/l}, &\text{ if } l \text{ divides } r,\\
0, &\text{ otherwise,}
\end{cases}
\quad\text{and}\quad
\widetilde{\rm Fr}_\zeta(\tqbinom{k_i}{r}) =
\begin{cases}
\binom{h_i}{r/l}, &\text{ if } l \text{ divides } r,\\
0, &\text{ otherwise.}
\end{cases}
\end{equation}

\vspace{-20pt}\hfill\qedsymbol
\end{thm}
Above, $\binom{h_i}{s}=\frac{h_i(h_i-1)\cdots (h_i-s+1)}{s!}$ for $s\ge 0$.

We shall denote by ${\rm Fr}_\zeta$ the restriction of $\widetilde{\rm Fr}_\zeta$ to \usres g. If $V$ is a $\tlie g$-module (resp. $\lie g$-module), denote by $\widetilde{\rm Fr}_\zeta^*(V)$ (resp. ${\rm Fr}_\zeta^*(V)$) the pull-back of $V$ by $\widetilde{\rm Fr}_\zeta$ (resp. ${\rm Fr}_\zeta$). Since $\widetilde{\rm Fr}_\zeta$ is a Hopf algebra map we have
\begin{equation}\label{e:Frontensor}
\widetilde{\rm Fr}_\zeta^*(V\otimes W)\cong \widetilde{\rm Fr}_\zeta^*(V)\otimes \widetilde{\rm Fr}_\zeta^*(W).
\end{equation}

Given $\lambda\in P^+$, there exist unique $\lambda'\in P^+_l,\lambda''\in P^+$ such that $\lambda=\lambda'+l\lambda''$. Also, for $\gb\lambda\in\cal P^+$, recall the decomposition $\gb\lambda = \gb\lambda'\phi_l(\gb\lambda'')$ given by \eqref{e:frobfactor}. The following theorem was proved in \cite{cproot,lroot}.

\begin{thm}\label{t:frobtensor}
Let $\lambda\in\cal P^+$ and $\gb\lambda\in\cal P^+$. Then:
\begin{enumerate}
\item $V_\zeta(\lambda) \cong V_\zeta(\lambda')\otimes V_\zeta(l\lambda'')$. Moreover,  $V_\zeta(l\lambda'')\cong {\rm Fr}_\zeta^*(V(\lambda''))$.
\item $V_\zeta(\gb\lambda) \cong V_\zeta(\gb\lambda')\otimes V_\zeta(\phi_l(\gb\lambda''))\cong V_\zeta(\phi_l(\gb\lambda''))\otimes V_\zeta(\gb\lambda')$. Moreover,  $V_\zeta(\phi_l(\gb\lambda''))\cong \widetilde{\rm Fr}_\zeta^*(V(\gb\lambda''))$.\hfill\qedsymbol
\end{enumerate}
\end{thm}

\begin{cor}\label{c:frev}
Let $\lambda\in P^+, a\in\mathbb C^\times$, and suppose $\lie g=\lie{sl}_{n+1}$. Then, $V_\zeta(l\lambda,a)\cong \widetilde{\rm Fr}_\zeta^*(V(\lambda,a^l))\cong V_\zeta(\phi_l(\gb\omega_{\lambda,a^l}))$.
\end{cor}

\begin{proof}
Using Proposition \ref{p:sevrep} and \eqref{e:skrpoly}, one easily sees that $V_\zeta(l\lambda,a)\cong V_\zeta(\phi_l(\gb\omega_{\lambda,a^l}))$ and we are done by Theorems \ref{t:frobtensor} and \ref{t:tev}.
\end{proof}

The following result is easily established.

\begin{prop}\label{p:tevsamep}
Let $\lambda,\mu\in P^+, a\in\mathbb C^\times$, and write $V(\lambda)\otimes V(\mu) = \opl_{j=1}^m V(\nu_j)$ for some $m\in\mathbb Z_{\ge 0}$ and $\nu_j\in P^+$. Then $V(\lambda,a)\otimes V(\mu,a)\cong \opl_{j=1}^m V(\nu_j,a)$.\hfill\qedsymbol
\end{prop}

By combining the above proposition with Theorem \ref{t:tev} we get:

\begin{cor}\label{c:tev}
Let $m\in\mathbb Z_{\ge 0}, \lambda_j\in P^+$, and $a_j\in\mathbb C^\times$ for $j=1,\dots,m$. Then $\otm_{j=1}^{m} V(\lambda_j,a_j)$ is irreducible iff $a_i\ne a_j$ for all $i\ne j$.\hfill\qedsymbol
\end{cor}

\begin{cor}\label{c:compredt}
Let $m\in\mathbb Z_{>0},\lambda,\lambda_j\in P^+, a,a_j\in\mathbb C^\times$ for $j=1,\dots,m$ and $\gb\lambda=\prod_{j=1}^m\gb\omega_{\lambda_j,a_j}$. We have:
\begin{enumerate}
\item If $a_i\ne a_j$ for all $i\ne j$, then $V_\zeta(\phi_l(\gb\lambda))\cong \otm_{j=1}^m V_\zeta(\phi_l(\gb\omega_{\lambda_j,a_j}))$.
\item If $V(\lambda)\otimes V(\lambda_1)\cong \opl_{k=1}^N V(\mu_j)$ for some $N\in\mathbb Z_{>0}$ and $\mu_j\in P^+$, then $V_\zeta(\phi_l(\gb\omega_{\lambda,a}))\otimes V_\zeta(\phi_l(\gb\omega_{\lambda_1,a}))$ $\cong \opl_{k=1}^N V_\zeta(\phi_l(\gb\omega_{\mu_j,a}))$.
\end{enumerate}
\end{cor}

\begin{proof}
Immediate from Proposition \ref{p:tevsamep}, Theorems \ref{t:frobtensor} and \ref{t:tev}, and Corollary \ref{c:frev}.
\end{proof}

The following theorem was proved in \cite{cp:qaa,cproot}.

\begin{thm}\label{t:qtev}
Let $\lie g=\lie{sl}_2, m,m'\in\mathbb Z_{\ge 0}, a_1,\dots,a_m,b_1,\dots,b_{m'}\in\mathbb C^\times, \lambda_1,\dots,\lambda_m\in P_l^+\backslash\{0\}$, and $\mu_1,\dots,\mu_{m'}\in P^+\backslash\{0\}$. Then, the tensor product
$$V_\xi(\lambda_1,a_1)\otimes\cdots\otimes V_\xi(\lambda_m,a_m)\otimes V_\xi(l\mu_1,b_1)\otimes\cdots\otimes V_\xi(l\mu_{m'},b_{m'})$$
is irreducible iff the polynomials $\gb\omega_{i,a_j,\lambda_j(h_i)}$ and  $\phi_l(\gb\omega_{\mu_k,b_k^l}), j=1,\dots,m,k=1,\dots m'$, are in general position. In particular, every simple object of $\wcal C_\xi$ is isomorphic to a tensor product of evaluation modules.\hfill\qedsymbol
\end{thm}

\section{Highest-$\ell$-weight tensor products}\label{s:hlwtensor}

\subsection{Main theorem and preliminary lemmas} The next theorem for the special case that $\xi$ has infinite order was the main result of \cite{vcbraid} and for $\xi=1$ it follows from Theorem \ref{t:tev}. The goal of this section is to prove it in the root of unity setting.

\begin{thm}\label{t:hlwtensor}
Let  $m\in\mathbb Z_{>0}$ and $\gb\lambda_j\in\cal P_\xi^+$ for $j=1,\dots,m$. If $(\gb\lambda_1,\dots,\gb\lambda_m)$ is in $\xi$-resonant order, then $V_\xi(\gb\lambda_1)\otimes\cdots\otimes V_\xi(\gb\lambda_m)$ is a highest-$\ell$-weight \ust_\xi g-module. Moreover, if $\xi=\zeta$ and $\gb\lambda_j\in\cal P_l^+$ for all $j=1,\dots,m$, then $V_\zeta(\gb\lambda_1)\otimes\cdots\otimes V_\zeta(\gb\lambda_m)$ is generated by the action of \ustfin g on the top weight space.
\end{thm}

The proof will be given in the following subsections. We first recall two important consequences of the theorem in the case $\xi$ has infinite order.

\begin{cor}\label{c:cyclic}
Let $\xi$ have infinite order. Given $a_1,\dots,a_k\in\mathbb C(\xi)$, there exists $\sigma\in S_k$ such that $V_\xi(\gb\omega_{i_1,a_{\sigma(1)},m_1})\otimes\cdots\otimes V_\xi(\gb\omega_{i_k,a_{\sigma(k)},m_k})$ is highest-$\ell$-weight for any choice of $i_1,\dots,i_k$ and of $m_1,\dots,m_k$.
\end{cor}

\begin{proof}
Immediate from Theorem \ref{t:hlwtensor} and Lemma \ref{l:cyclic}.
\end{proof}

This corollary was used in \cite{cmq} together with Theorem \ref{t:lwmdim} for proving the following theorem.

\begin{thm}\label{t:wtfund}
Suppose $\xi$ has infinite order and that $\gb\lambda=\prod_{j=1}^m \gb\omega_{i_j,a_j}\in\cal P_\xi^+$ for some $i_j\in I, a_j\in\mathbb C(\xi)^\times$. Then, there exists $\sigma\in S_m$ such that $W_\xi(\gb\lambda)\cong V_\xi(\gb\omega_{i_{\sigma(1)},a_{\sigma(1)}})\otimes\cdots\otimes V_\xi(\gb\omega_{i_{\sigma(m)},a_{\sigma(m)}})$.\hfill\qedsymbol
\end{thm}

\begin{rem}
Corollary \ref{c:cyclic} and Theorem \ref{t:wtfund} are false if $\xi=\zeta$ (cf. \cite[Remark 7.18]{vava:standard}). Such different behavior in the root of unity setting stems from two sources. One is related to the possible lack of $\zeta$-regularity of the highest $\ell$-weight which can happen  for any $\lie g$ (see Corollary \ref{c:wtfundr} and the discussion after Corollary \ref{c:wtfundrgen}). The second is the possible reducibility of the Weyl modules whose highest $\ell$-weigh is a fundamental $\ell$-weight. This does not occur for $\lie g$ of type $A$ because every fundamental weight is minuscule. For an example of the latter situation see Example \ref{ex:dn} below and the remark that follows it.
\end{rem}

Before proceeding with the proof of Theorem \ref{t:hlwtensor}, we give a counterexample showing that the ``loop'' analogue of the first statement of Proposition \ref{p:uqlat}(a) does not hold in general.

\begin{ex}\label{ex:uqtlat}
Let $\lie g=\lie{sl}_2, I=\{i\}$ and $a,b\in\mathbb A^\times$ be such that $\frac{a}{b}\ne q^{-2}$. It follows that $V=V_q(\gb\omega_{i,a})\otimes V_q(\gb\omega_{i,b})$ is a highest-$\ell$-weight module with highest $\ell$-weight $\gb\lambda=\gb\omega_{i,a}\gb\omega_{i,b}\in\cal P_\mathbb A^s$. Let $v$ and $w$ be highest-$\ell$-weight vectors of $V_q(\gb\omega_{i,a})$ and $V_q(\gb\omega_{i,b})$, respectively, and set $v_1=x_i^-v, w_1=x_i^-w$. Then $v_1\in V_q(\gb\omega_{i,a})_{\gb\omega_{i,aq^2}^{-1}}, w_1\in V_q(\gb\omega_{i,b})_{\gb\omega_{i,bq^2}^{-1}}$. Using Proposition \ref{p:comultip}(d) one easily sees that \begin{equation}\label{e:lwvuqtlat}
v_1\otimes w\in V_{\gb\omega_{i,b}\gb\omega_{i,aq^2}^{-1}} \qquad\text{and}\qquad v\otimes w_1+c v_1\otimes w\in V_{\gb\omega_{i,a}\gb\omega_{i,bq^2}^{-1}}
\end{equation}
for some  $c\in\mathbb C(q)$ (notice that Proposition \ref{p:comultip}(d) does not allow us to compute $c$ precisely). It follows from \eqref{e:lwvuqtlat} that, if $a=b$, any weight vector of $V$ is an $\ell$-weight vector and, hence, this would not be a counterexample. Thus, assume from now on that $a\ne b$.

Set also $\gbr v_1=x_i^-(v\otimes w), \gbr v_2=x_{i,1}^-(v\otimes w), \gbr w_1=v\otimes w_1+c v_1\otimes w, \gbr w_2= v_1\otimes w$, and $L=\uqtres g(v\otimes w)$. It is known that $\{\gbr v_1,\gbr v_2\}$ is an $\mathbb A$-basis for the zero weight space $L_0:=L\cap V_0$.
Using Lemma \ref{l:sl2actev} and Proposition \ref{p:comultip}, one easily computes that
\begin{equation}\label{e:uqtlat}
\gbr v_1 = q^{-1}v_1\otimes w +v\otimes w_1 \qquad\text{and}\qquad \gbr v_2=aq^2v_1\otimes w+ bq v\otimes w_1.
\end{equation}

We will show that there exists no $\mathbb A$-basis of $L_0$ formed by $\ell$-weight vectors provided $(b-aq^2)\notin \mathbb A^\times$. Since there are plenty of choices for $a,b\in\mathbb A^\times$ such that $b-aq^2\notin \mathbb A^\times$, we get our counter example.

Using \eqref{e:lwvuqtlat}, we see that any such basis would be of the form $\{\alpha\gbr w_1,\beta\gbr w_2\}$ for some $\alpha,\beta\in\mathbb C(q)$.
By contradiction, suppose there exist $\alpha,\beta\in\mathbb C(q)$ such that $\{\alpha\gbr w_1,\beta\gbr w_2\}$  is an $\mathbb A$-basis of $L_0$. Then, the matrix $A$ whose columns are the coordinates of $\gbr v_1,\gbr v_2$ with respect to this basis must have entries in $\mathbb A$ and its determinant must be in $\mathbb A^\times$. Using \eqref{e:uqtlat} one easily sees that
\begin{equation}
A=\begin{pmatrix}
\alpha^{-1}& qb\alpha^{-1}\\ (q^{-1}-c)\beta^{-1}&(aq-c\beta)q\beta^{-1}
\end{pmatrix}
\qquad\text{and}\qquad \det(A) = (\alpha\beta)^{-1}(aq^2-b).
\end{equation}
It follows that $\alpha^{-1},\beta^{-1}\in\mathbb A$ and, hence, $\det(A)\in\mathbb A^\times$ iff $\alpha,\beta,(b-aq^2)\in\mathbb A^\times$.
\end{ex}


For the next two lemmas we will use the following notation. Let $\gb\lambda,\gb\mu\in\cal P_\xi^+, \lambda=\wt(\gb\lambda),\mu=\wt(\gb\mu)$, and $v_{_\gb\lambda}, v_{_\gb\mu}$ be highest-$\ell$-weight vectors of highest-$\ell$-weight modules $V,W$ of highest $\ell$-weight $\gb\lambda$ and $\gb\mu$, respectively.  Let also $v_{_\gb\lambda}^w$ be a generator of $V_{w\lambda}$ for $w\in\cal W$.

\begin{lem}\label{l:TqW}
 Suppose $i\in I$ and $w\in\cal W$ are such that $\ell(s_iw)=\ell(w)+1$. Then, the \ust_\xi{g_i}-module $\ust_\xi{g_i}(v_{_\gb\lambda}^w\otimes v_{_\gb\mu})$ is a quotient of $W_\xi((T_w\gb\lambda)_i\gb\mu_i)$.
\end{lem}

\begin{proof}
Straightforward using Proposition \ref{p:comultip} and Corollary \ref{c:Twv} (cf. \cite[Lemma 4.3]{vcbraid} for $\xi$ not a root of unity).
\end{proof}

\begin{lem}\label{l:lth=all}
Suppose $v_{_\gb\lambda}^{w_0}\otimes v_{_\gb\mu}\in\ust_\xi g(v_{_\gb\lambda}\otimes v_{_\gb\mu})$. Then, $V\otimes W = \ust_\xi g(v_{_\gb\lambda}\otimes v_{_\gb\mu})$. Moreover, if $\xi=\zeta, V=\ustfin g v_{_\gb\lambda}$,  $W=\ustfin g v_{_\gb\mu}$, and $v_{_\gb\lambda}\in \ustfin{n^+} v_{_\gb\lambda}^{w_0}$, then  $V\otimes W = \ustfin g(v_\gb\lambda\otimes v_\gb\mu)$.
\end{lem}

\begin{proof}
The proof is a refinement of that of \cite[Lemma 4.2]{vcbraid}. Since $(x_{i,r}^-)^{(k)}v_{_\gb\lambda}^{w_0}=0$ for all $i\in I, r,k\in\mathbb Z, k>0$, it follows from Proposition \ref{p:comultip} that $(x_{i,r}^-)^{(k)}(v_{_\gb\lambda}^{w_0}\otimes v_{_\gb\mu}) = v_{_\gb\lambda}^{w_0}\otimes ((x_{i,r}^-)^{(k)}v_{_\gb\mu})$. Hence, $v_{_\gb\lambda}^{w_0}\otimes W\subseteq \ust_\xi g(v_{_\gb\lambda}\otimes v_{_\gb\mu})$. Similarly, if $\xi=\zeta, V=\ustfin g v_{_\gb\lambda}$, and $W=\ustfin g v_{_\gb\mu}$, then $v_{_\gb\lambda}^{w_0}\otimes W\subseteq \ustfin g(v_{_\gb\lambda}\otimes v_{_\gb\mu})$.

Next we prove that $v_{_\gb\lambda}\otimes W\subseteq \ust_\xi g(v_{_\gb\lambda}\otimes v_{_\gb\mu})$ and, under the hypothesis of the second statement, that $v_{_\gb\lambda}\otimes W\subseteq \ustfin g(v_{_\gb\lambda}\otimes v_{_\gb\mu})$. Since $v_{_\gb\lambda}\in \ust_\xi{n^+}v_{_\gb\lambda}^{w_0}$ (respectively, $v_{_\gb\lambda}\in \ustfin{n^+}v_{_\gb\lambda}^{w_0}$), it suffices to show that
\begin{equation}\label{e:lth=all}
\left((x_{i_1,r_1}^+)^{(k_1)}\cdots (x_{i_m,r_m}^+)^{(k_m)}\ v_{_\gb\lambda}^{w_0}\right)\otimes W \subseteq \ust_\xi g(v_{_\gb\lambda}\otimes v_{_\gb\mu})
\end{equation}
for all $m\in\mathbb Z_{\ge 0}, k_j\in\mathbb Z_{>0}, i_j\in I, r_j\in\mathbb Z$ (for the second statement it suffices to prove this with $k_j=1$ for all $j=1,\dots,m$ and with \ustfin{g} in place of \ust_\xi g). We proceed by induction on the height of $\sum_j k_j\alpha_{i_j}$. Induction clearly starts for $m=0$. We shall write the proof for the first statement only since the second one is proved similarly.

Let $v=(x_{i_2,r_2}^+)^{(k_2)}\cdots (x_{i_m,r_m}^+)^{(k_m)}\ v_{_\gb\lambda}^{w_0}$ and assume, by induction hypothesis, that $v\otimes W\subseteq$ \ust_\xi g $(v_{_\gb\lambda}\otimes v_{_\gb\mu})$. We now prove that $((x_{i_1,r_1}^+)^{(k_1)}v)\otimes W\subseteq \ust_\xi g(v_{_\gb\lambda}\otimes v_{_\gb\mu})$. We consider only the case $r_1\ge 0$ since the case $r_1<0$ is similar. The proof is by  induction on $k_1$ with a further subinduction on $r_1$. Let $v'\in W$. By Proposition \ref{p:comultip} we have
\begin{align*}
(x_{i_1,r_1}^+)^{(k_1)}(v\otimes v') = ((x_{i_1,r_1}^+)^{(k_1)}v)\otimes v' + \varpi+\varpi'
\end{align*}
where $\varpi$ is a sum of vectors which belong to
$$\left((x_{i_1',r_1'}^+)^{(k_1')}\cdots (x_{i_{m'}',r_{m'}'}^+)^{(k_{m'}')}\ v_{_\gb\lambda}^{w_0}\right)\otimes W$$
with the height of $\sum_{j=1}^{m'} k_j'\alpha_{i_j'}$ strictly smaller than that of $\sum_{j=1}^m k_j\alpha_{i_j}$, while $\varpi'$ is a sum of vectors belonging to
$$\left((x_{i_1,r_1}^+)^{(k)}(x_{i_1,s_1}^+)^{(k_1')}\cdots (x_{i_1,s_{m'}}^+)^{(k_{m'}')}\ v\right)\otimes W$$
with $0\le k<k_1, 0\le s_j<r_1$ and $k+\sum_{j=1}^{m'} k_j'=k_1$. Hence, by the induction hypothesis on the height we have $\varpi\in \ust_\xi g(v_{_\gb\lambda}\otimes v_{_\gb\mu})$, while the same is true for $\varpi'$ by the subinduction hypothesis on $k_1$ and $r_1$ (observe that if $k_1=1$ then $k=0$ and if $r_1=0$ then $\varpi'=0$ which shows that both subinductions start). Since $(x_{i_1,r_1}^+)^{(k_1)}(v\otimes v')$ obviously belongs to $\ust_\xi g(v_{_\gb\lambda}\otimes v_{_\gb\mu})$, it follows that $((x_{i_1,r_1}^+)^{(k_1)}v)\otimes v'\in \ust_\xi g(v_{_\gb\lambda}\otimes v_{_\gb\mu})$ for all $v'\in W$. This completes the proof of \eqref{e:lth=all}.

It now suffices to show that
\begin{equation}\label{e:lth=allc}
\left((x_{i_1}^{\epsilon_1})^{(k_1)}\cdots (x_{i_m}^{\epsilon_m})^{(k_m)}\ v_{_\gb\lambda}\right)\otimes W \subseteq \ust_\xi g(v_{_\gb\lambda}\otimes v_{_\gb\mu})
\end{equation}
for all $m\in\mathbb Z_{\ge 0}, k_j\in\mathbb Z_{>0}, i_j\in\hat I$, and $\epsilon_j\in\{+,-\}$.  The proof of \eqref{e:lth=allc} is identical to the proof of \eqref{e:lth=all} replacing $(x_{i_j,r_j}^+)^{(k_j)}$ by $(x_{i_j}^{\epsilon_j})^{(k_j)}$ and $v_{_\gb\lambda}^{w_0}$ by $v_{_\gb\lambda}$.
\end{proof}

The next lemma is easily established.

\begin{lem}\label{l:ressl2}
Let $\gb\lambda\in\cal P_\xi^+$ and let $v$ be a highest-$\ell$-weight vector of $V_\xi(\gb\lambda)$. Then $\ust_\xi{g_i}v$ is an irreducible \ust_\xi{g_i}-module for all $i\in I$.\hfill\qedsymbol
\end{lem}

\subsection{The $\lie{sl}_2$-case}

Throughout this subsection we let $\lie g=\lie{sl}_2$ and write $I=\{i\}$.

\begin{proof}[Proof of Theorem \ref{t:hlwtensor} for $\lie g=\lie{sl}_2$]
If $\xi=1$ this is part of Theorem \ref{t:tev} and if $\xi$ has infinite order this follows from \cite[Lemma 4.10]{cp:qaa}. Thus, assume $\xi=\zeta$. We first prove  the second statement.  We proceed by induction on $m$ observing that there is nothing to prove when $m=1$. Let $v_j, j=1,\dots,m$, be highest-$\ell$-weight vectors for $V_\zeta(\gb\lambda_j)$ and $v'=v_2\otimes\cdots\otimes v_m$ . By the induction hypothesis, $V'=\ustfin g v'$.

The proof runs parallel to that of \cite[Lemma 4.10]{cp:qaa}. Let $I=\{i\}$. By Theorem \ref{t:qtev}, it suffices to consider the case $\gb\lambda_j = \gb\omega_{i,a_j,n_j}$ for all $j=1,\dots,m$ and some $a_j\in\mathbb C^\times$ and $0<n_j<l$. Given $k\in\{0,\dots,n_j\}$, set $v_j^k=(x_i^-)^kv_j$ so that $\{v_j^k:k=0,\dots,n_1\}$ is a basis of $V(\gb\lambda_j)$. By Lemma \ref{l:lth=all}, it suffices to show that
\begin{equation}\label{e:hlwtensorsl2}
v_1^{k+1}\otimes v'\in \ustfin g(v_1^k\otimes v') \quad\text{for all}\quad k=0,\dots,n_1-1.
\end{equation}
It is not difficult to see, using Proposition \ref{p:comultip}, that there exist $a_{r,s}\in\mathbb C$ such that
\begin{equation}
x_{i,r}^-(v_1^k\otimes v') = \sum_{s=1}^m a_{r,s} w_s
\end{equation}
where $w_1 = v_1^{k+1}\otimes v', w_2=v_1^k\otimes v_2^1\otimes v_3\otimes\cdots\otimes v_m, \dots, w_m=v_1^k\otimes v_2\otimes\cdots\otimes v_{m-1}\otimes v_m^1$. In fact, $a_{r,s}$ can be explicitly computed using Lemma \ref{l:sl2actev}. It turns out that the determinant of the matrix $(a_{r,s}), 0\le r,s\le m$, is different from zero iff $a_1/a_j\ne \zeta^{-(n_1+n_j-2k)}$ and $a_j/a_\ell\ne \zeta^{-(n_j+n_\ell)}$ for all $1<j<\ell\le m$ (see the proof of \cite[Lemma 4.10]{cp:qaa}). Since $(\gb\lambda_1,\dots,\gb\lambda_m)$ is in $\xi$-resonant order, \eqref{e:hlwtensorsl2} follows.

To prove the first statement, write $\gb\lambda_j = \gb\lambda_j'\phi_l(\gb\lambda_j'')$ as in \eqref{e:frobfactor} and notice that Theorem \ref{t:frobtensor}(b) and Corollary \ref{c:compredt}(a) imply that
$$V_\zeta(\gb\lambda_1)\otimes\cdots\otimes V_\zeta(\gb\lambda_m) \cong V_\zeta(\gb\lambda_1')\otimes\cdots\otimes V_\zeta(\gb\lambda_m')\otimes V_\zeta(\phi_l(\gb\mu))$$
where $\gb\mu = \prod_{j=1}^m \gb\lambda_j''$. Set $V=V_\zeta(\gb\lambda_1')\otimes\cdots\otimes V_\zeta(\gb\lambda_m')$ and $\gb\lambda=\prod_{j=1}^m\gb\lambda_j'$. Since we have already proved the ``moreover'' part of the theorem, it follows that $V$ is a highest-$\ell$-weight module of highest-$\ell$-weight $\gb\lambda$. Let $v_{_\gb\lambda}, v_{_\gb\mu}$ be highest-$\ell$-weight vectors of $V$ and $W=V_\zeta(\phi_l(\gb\mu))$, respectively. By the second statement of Theorem \ref{t:frobtensor}(b), $x_{i,r}^-$ acts trivially on $W$ and, hence, $V\otimes v_{_\gb\mu}\in\ustfin g(v_{_\gb\lambda}\otimes v_{_\gb\mu})$. We are done by Lemma \ref{l:lth=all}.
\end{proof}

\begin{cor}\label{c:hlwtfund}
Let $m\in\mathbb Z_{>0}, a_1,\dots, a_m\in\mathbb C(\xi)^\times$. Then $V=V_\xi(\gb\omega_{i,a_1})\otimes \cdots\otimes V_\xi(\gb\omega_{i,a_m})$ is highest-$\ell$-weight iff $a_j/a_k\ne \xi^{-2}$ for all $j<k$.
\end{cor}

\begin{proof}
The ``if'' direction is a particular case of the previous proposition while the ``only if'' direction is a corollary of the proof. In fact, suppose $a_s/a_{s'}=\xi^{-2}$ for some $s,s'$, let $v_j$ be a highest-$\ell$-weight vector of $V(\gb\omega_{i,a_j}), v=v_1\otimes\cdots v_m$, and $\gb\lambda=\prod_{j=1}^m \gb\omega_{i,a_j}$. Let also $w$ be a highest-$\ell$-weight vector for $W_\xi(\gb\lambda)$. It follows from the proof of Theorem \ref{t:weylm} that $W_\xi(\gb\lambda)_{(m-2)\omega_i}$ is spanned by $\{x_{i,r}^-w: r=0,\dots,m-1\}$. Let $V'$ be the submodule of $V$ generated by $v$. Since $V'$ is a quotient of $W_\xi(\gb\lambda)_{(m-2)\omega_i}$ and $\dim(V_{(m-2)\omega_i})=m$, it suffices to show that the set $\{x_{i,r}^-v: r=0,\dots,m-1\}$ is linearly dependent. But this was done as part of the above proof for the $\lie{sl}_2$-case of Theorem \ref{t:hlwtensor}.
\end{proof}

\begin{cor}\label{c:wtfundr}
Let $\gb\lambda\in\cal P_\xi^+$. Then, $W_\xi(\gb\lambda)$ is isomorphic to a tensor product of fundamental representations iff $\gb\lambda$ is $\xi$-regular. Moreover, in that case, $W_\xi(\gb\lambda)=\ustfin{sl_2}W_\xi(\gb\lambda)_{\wt(\gb\lambda)}$.
\end{cor}

\begin{proof}
Write $\gb\lambda=\prod_{j=1}^m\gb\omega_{i,a_j}$. Then $W_\xi(\gb\lambda)$ is a tensor product of fundamental representations iff it is isomorphic to $V_\xi(\gb\omega_{i,a_{\sigma(1)}})\otimes \cdots\otimes V_\xi(\gb\omega_{i,a_{\sigma(m)}})$ for some $\sigma\in S_m$. By Theorem \ref{t:lwmdim}, the dimension of such tensor products coincides with that of  $W_\xi(\gb\lambda)$. Hence, an isomorphism exists iff the tensor product is highest-$\ell$-weight. If $\gb\lambda$ is not $\xi$-regular, then the previous corollary and Lemma \ref{l:regular} imply that, for any choice of $\sigma$, such tensor product is not highest-$\ell$-weight and, hence, cannot be isomorphic to $W_\xi(\gb\lambda)$. Conversely, if $\gb\lambda$ is $\xi$-regular, then the previous corollary and Lemma \ref{l:regular} imply that  there exists $\sigma$ for which such tensor product is highest-$\ell$-weight.

The second statement follows from the $\lie{sl}_2$-case of Theorem \ref{t:hlwtensor} since $\gb\omega_{i,a}\in \cal P_l^+$ for all $a$.
\end{proof}

\subsection{Proof of Theorem \ref{t:hlwtensor} for roots of unity in the general case}
The first statement is deduced from the second in the same manner as done for the $\lie{sl}_2$-case in the previous subsection. The proof of the second statement runs parallel to the one for the generic case with a few modifications. We again proceed by induction on $m$ and observe that there is nothing to prove when $m=1$. Once more, let $v_j, j=1,\dots,m$ be highest-$\ell$-weight vectors for $V_\zeta(\gb\lambda_j)$ and $v'=v_2\otimes\cdots\otimes v_m$. By the induction hypothesis, $V'=\ustfin g v'$.

Let $s_{i_N}\cdots s_{i_1}$ be a reduced expression for $w_0$ satisfying the properties in the definition of $\zeta$-regularity. By Lemma \ref{l:lth=all}, it suffices to show that
\begin{equation}\label{e:hlwtensorr}
v_1^{s_{i_j}\cdots s_{i_1}}\otimes v' \in \ustfin{g_{i_j}} (v_1^{s_{i_{j-1}}\cdots s_{i_1}}\otimes v') \quad\text{for all}\quad j=1,\dots,N.
\end{equation}
Since $\gb\lambda_k\in \cal P_l^+$, Proposition \ref{p:lufinir} and Lemma \ref{l:ressl2} imply that $\ustfin {g_{i_j}}v_k=\ustfin {g_{i_j}}v_k$ is an irreducible \ustfin {g_{i_j}}-module of highest-$\ell$-weight $(\gb\lambda_k)_{i_j}$. By condition (a) of the definition of $\zeta$-resonant ordering, the $(m-1)$-tuple $((\gb\lambda_2)_{i_j}, \dots, (\gb\lambda_m)_{i_j})$ of polynomials is in $\zeta_{i_j}$-resonant order. In particular, it follows from the $\lie{sl}_2$-case that
$$\ustfin{g_{i_j}}v' = (\ustfin{g_{i_j}} v_2)\otimes\cdots\otimes (\ustfin{g_{i_j}} v_m).$$
On the other hand, $\ustres {g_{i_j}}v_1^{s_{i_{j-1}}\cdots s_{i_1}}$ is a quotient of $W_\zeta((T_{i_{j-1}}\cdots T_{i_1}\gb\lambda_1)_{i_j})$ by Lemma \ref{l:TqW}. By condition (c) of the definition of $\zeta$-resonant ordering, the polynomial $(T_{i_{j-1}}\cdots T_{i_1}\gb\lambda_1)_{i_j}$ is $\zeta_{i_j}$-regular. Therefore, by Corollary \ref{c:wtfundr}, $W_\zeta((T_{i_{j-1}}\cdots T_{i_1}\gb\lambda_1)_{i_j})$ is isomorphic to a tensor product of the form $V_\zeta(\gb\omega_{i_j,a_1})\otimes\cdots\otimes V_\zeta(\gb\omega_{i_j,a_k})$
for some $a_1,\dots,a_k\in\mathbb C^\times$ and where $k=\wt(T_{i_{j-1}}\cdots T_{i_1}\gb\lambda_1)(h_{i_j})$. Here the ordering of the tensor product is chose so that $(\gb\omega_{i_j,a_1},\dots,\gb\omega_{i_j,a_k})$ is a tuple of polynomials in $\zeta_{i_j}$-resonant ordering. Moreover, by the $\lie{sl}_2$-case, this tensor product is generated by the action of $\ustfin{g_{i_j}}$ on the highest weight space. Condition (b) of the definition of $\zeta$-resonant ordering says that tuple of polynomials $\left((T_{i_{j-1}}\cdots T_{i_1}\gb\lambda_1)_{i_j}, (\gb\lambda_2)_{i_j}, \dots, (\gb\lambda_m)_{i_j}\right)$ is in $\zeta_{i_j}$-resonant order. One easily checks that this implies that $(\gb\omega_{i_j,a_1},\dots,\gb\omega_{i_j,a_k},$ $(\gb\lambda_2)_{i_j}, \dots, (\gb\lambda_m)_{i_j})$
is also a tuple of polynomials in $\zeta_{i_j}$-resonant order. It then follows from the $\lie{sl}_2$-case  that
$$\ustfin{g_{i_j}} (v_1^{s_{i_{j-1}}\cdots s_{i_1}}\otimes v') = (\ustfin{g_{i_j}} v_1^{s_{i_{j-1}}\cdots s_{i_1}})\otimes (\ustfin{g_{i_j}}v')$$
which implies \eqref{e:hlwtensorr}.\hfill\qedsymbol

\begin{cor}\label{c:wtfundrgen}
Suppose $\gb\lambda=\prod_{j=1}^m\gb\omega_{i_j,a_j}\in\cal P_\xi^+$ is $\xi$-regular and that $W_\xi(\gb\omega_{i_j,a_j})$ is irreducible for all $j=1,\dots,m$. Then, $W_\xi(\gb\lambda)$ is isomorphic to a tensor product of fundamental representations and $W_\xi(\gb\lambda)=\ustfin{g}W_\xi(\gb\lambda)_{\wt(\gb\lambda)}$.
\end{cor}

\begin{proof}
The proof is identical to that of Corollary \ref{c:wtfundr}.
\end{proof}

The hypothesis that $W_\xi(\gb\omega_{i_j,a_j})$ is irreducible in the above corollary is well-known to be vacuous if $\xi$ has infinite order and, hence, is needed only for $\xi$ of finite order. Regardless of the value of $\xi$, the hypothesis is also clearly vacuous whenever $\omega_{i_j}$ is a minuscule weight. We will see in \S\ref{ss:charfund} below that in the root of unity case $W_\xi(\gb\omega_{i_j,a_j})$ may be reducible. In the $\lie{sl}_2$-case we proved the converse of the above theorem (Corollary \ref{c:wtfundr}). One of the reasons that made this possible is that Corollary \ref{c:hlwtfund} establishes a necessary and sufficient condition for a tensor product of fundamental representations to be highest-$\ell$-weight while Theorem \ref{t:hlwtensor} gives only a sufficient condition even for $\lie g=\lie{sl}_2$ as it is clear from Theorem \ref{t:qtev}. In fact, Theorem \ref{t:qtev} makes it natural to wonder if weak resonant ordering is a sufficient and necessary condition in general. It is clear from the proof above that, in order to prove Theorem \ref{t:hlwtensor} with weak resonant order in place of resonant ordering, it suffices to prove it for $\lie g=\lie{sl}_2$. The latter is true for a two-fold tensor product of evaluation modules.

\begin{rem}
Theorem \ref{t:hlwtensor} provides an algorithm for deciding if the simple module associated to a given dominant $\ell$-weight $\gb\lambda$ is not prime. Namely, if $\gb\lambda=\gb\mu\gb\nu$ for some nontrivial  $\gb\mu,\gb\nu\in\cal P_\xi^+$ such that both $(\gb\mu,\gb\nu)$ and $(\gb\nu,\gb\mu)$ are in $\xi$-resonant order, then $V_\xi(\gb\lambda)$ is isomorphic to $V_\xi(\gb\mu)\otimes V_\xi(\gb\nu)$ and, hence, is not prime (one uses the well-known fact that if a tensor product is highest-$\ell$-weight in any order, than it is irreducible and the isomorphism class does not depend on the order).
\end{rem}

\section{Blocks and Characters of Fundamental Modules}\label{s:block}

\subsection{On Jordan-H\"older constituents}\label{ss:JHconst}
We will need several results regarding the irreducible constituents of \ust_\xi g-modules obtained by specializing \uqt g-modules. We begin by recalling the following well-known fact.

\begin{prop}\label{p:constensor}
Let $\cal C$ be a Jordan-H\"older tensor category and $V,W$ objects in $\cal C$. The set of irreducible constituents of $V\otimes W$ (counted with multiplicities) is the union of the sets of irreducible constituents of $V_i\otimes W_j$ where $V_i$ runs through the irreducible constituents of $V$ and $W_j$ runs through the irreducible constituents of $W$.\hfill\qedsymbol
\end{prop}

Recall the definitions of $I_\bullet\subseteq I$ from \S\ref{ss:basicnot} (Table 1) and of $\cal P_{\xi,I_\bullet}^+$ given at the end of \S\ref{ss:ec}.

\begin{thm}\label{t:generator}
Every simple object of $\cal C_\xi$ is isomorphic to the quotient of a submodule of a tensor product of the modules $W_\xi(\omega_i)$ for $i\in I_\bullet$.
\end{thm}

\begin{proof}
For simplicity, we write the proof of the theorem for the case that $I_\bullet$ is a singleton, i.e., that $\lie g$ is not of type $D_{2m}$. Thus, let $i$ denote the unique element of $I_\bullet$.
For $\xi=1$ the result is well-known and the proof for $\xi$ not a root of unity is analogous. Moreover, every simple object of $\cal C_\xi$ is an irreducible summand of a tensor power of $V_\xi(\omega_i)=W_\xi(\omega_i)$. We now consider the case $\xi=\zeta$.

Let $\lambda\in P^+$ and $m\in\mathbb Z_{\ge 0}$ be such that $V_q(\lambda)$ is a summand of $V_q(\omega_i)^{\otimes m}$. Fix highest-weight vectors $v_j$ of the $j$-th factor of this tensor product, let $L_j = \uqres g v_j$, $j=1,\dots,m$, and $L=L_1\otimes\cdots\otimes L_m$. Quite clearly $\overline L\cong \otimes_j\overline{L_j}\cong W_\zeta(\omega_i)^{\otimes m}$. Hence, it suffices to show that there exists $v\in L$ such that $x_k^+v=0$ for all $k\in I$, $\uq g v\cong V_q(\lambda)$, and the image $\bar v$ of $v$ in $\overline{L}$ is nonzero.

The $\mathbb A$-module $L_\lambda$ has an $\mathbb A$-basis formed by elements of the form
$$\left( (x_{i_{1,1}}^-)^{(k_{1,1})}\cdots (x_{i_{r_1,1}}^-)^{(k_{r_1,1})} v_1\right) \otimes\cdots\otimes \left( (x_{i_{1,m}}^-)^{(k_{1,m})}\cdots (x_{i_{r_m,m}}^-)^{(k_{r_m,m})} v_m\right):= v_{\vec{i},\vec{k}}.$$
Here $\vec{i}=(i_{1,1},\cdots,i_{r_1,1};\cdots; i_{1,m},\cdots,i_{r_m,m})$ and similarly for $\vec{k}$. It easily follows that there is an $\mathbb A$-linear combination of such vectors, say $v'$, satisfying the two first properties required of $v$. Write
$$v'=\sum_{\vec i,\vec k} c_{\vec i,\vec k}\ v_{\vec i,\vec k}, \quad c_{\vec i,\vec k}\in\mathbb A,$$
and observe that there exist $c'_{\vec i,\vec k}\in\mathbb C[q]$ and $f\in\mathbb A$ such that the nonzero $c'_{\vec i,\vec k}$ are relatively prime in $\mathbb C[q]$ and
$$v' = f\sum_{\vec i,\vec k} c'_{\vec i,\vec k}\ v_{\vec i,\vec k}:= fv.$$
Since the $c'_{\vec i,\vec k}$ are relatively prime, it follows that $\epsilon_\zeta(c'_{\vec i,\vec k})\ne 0$ for at least one value of the pair $(\vec i,\vec k)$. The theorem follows.
\end{proof}

In order to prove an affine analogue of the above theorem we will need the following proposition which will also be used to prove the main result of \S\ref{ss:blocks}.

\begin{prop}\label{p:constspec}
Let $\gb\lambda\in\cal P_\mathbb A^s$, $V$ a nontrivial quotient of $W_q(\gb\lambda)$, and $\xi\in\mathbb C'\backslash\{q\}$. If $V_q(\gb\mu)$ is an irreducible constituent of $V$ with multiplicity $m$, then $V_\xi(\overline{\gb\mu})$ is an irreducible constituent of $\overline V$ with multiplicity at least $m$.
\end{prop}

\begin{proof}
By an obvious induction on the length of $V$, it suffices to show the proposition in the case $V_q(\gb\mu)$ is a submodule of $V$. The argument is essentially the same as that of \cite[Proposition 4.16]{jm:hla}. Namely, let $v$ be a highest-$\ell$-weight vector of $V$ and $L=\uqtres gv$. Then, from the proof of Theorem \ref{t:uqtlat},  $L$ has an $\mathbb A$-basis consisting of vectors which are $\mathbb A$-linear combinations of elements of the form $(x_{\alpha_{i_1},r_1}^-)^{(k_1)}\cdots (x_{\alpha_{i_s},r_s}^-)^{(k_s)}v$. Let $v_1,\cdots, v_r$ be an $\mathbb A$-basis for $L_\mu$ where $\mu=\wt(\gb\mu)$. Any highest-$\ell$-weight vector for $V_q(\gb\mu)$ is a solution $\sum_{j=1}^r c_jv_j$, for some $c_j\in\mathbb C(q)$, of the linear system
$$(x_{i,s}^+)^{(k)}\left(\sum_{j=1}^r c_jv_j\right)=0, \qquad\Lambda_{i,s}\left(\sum_{j=1}^r c_jv_j\right)=\gb\Psi_\gb\mu(\Lambda_{i,s})\left(\sum_{j=1}^r c_jv_j\right)$$
for all $i\in I, s\in\mathbb Z, k\in\mathbb Z_{> 0}$.  By assumption, there exists a nontrivial solution for this system. Since $L$ is admissible and the $\ell$-weights of $V$ are in $\cal P_\mathbb A$ (by Theorem \ref{t:cone} given that  $\gb\lambda\in\cal P_\mathbb A^s$), it follows that there exists a solution with the $c_j$ lying in $\mathbb A$. Hence, there is also a solution with $c_j\in\mathbb C[q]$ and such that the nonzero $c_j$ are relatively prime. This completes the proof similarly to that of Theorem \ref{t:generator}.
\end{proof}

\begin{thm}\label{t:lgenerator}
For every $\gb\lambda\in\cal P_\xi^+$, there exists $\gb\mu\in\cal P_{\xi,I_\bullet}^+$ such that $V_\xi(\gb\lambda)$ is an irreducible constituent of $W_\xi(\gb\mu)$. Moreover, if $\gb\lambda\in\cal P_\mathbb A^s$, then $\gb\mu$ can be chosen to be in $\cal P_\mathbb A^s$ as well.
\end{thm}

\begin{proof}
For $\xi=q$, it follows from the results of \cite{cp:dorey,em} that $V_q(\gb\lambda)$ is a constituent of a tensor product of the form $V_q(\gb\omega_{i_1,a_1})\otimes\cdots V_q(\gb\omega_{i_m,a_m})$ with $i_j\in I_\bullet$ and $a_j\in\mathbb C^\times$ ($a_j\in\mathbb A^\times$ if $\gb\lambda\in\cal P_\mathbb A^s$). The theorem now follows in this case by Corollaries \ref{c:cyclic} and \ref{c:constensor}.

For $\xi\in\mathbb C'$, consider the module $V_q(\gb\lambda)$ and apply the theorem for the case $\xi=q$. Thus, let $\gb\mu\in\cal P_{\xi,I_\bullet}^+\cap\cal P_\mathbb A^s$ be such that $V_q(\gb\lambda)$ is a constituent of $W_q(\gb\mu)$. By Proposition \ref{p:constspec}, $V_\xi(\gb\lambda)$ is a constituent of $\overline{W_q(\gb\mu)}$.
\end{proof}

\subsection{Blocks}\label{ss:blocks}

Recall the notation introduced in \S\ref{ss:ec}. We shall need the following proposition which follows from the results of \cite{cmq,em}.

\begin{prop}\label{p:trivconst}
Let $\gb\mu\in\cal P_{q,I_\bullet}^+$ and $\gb\lambda=\gb\mu\gb\tau_{q,k,a}$ for some $k\in\{1,2,3\}$ and some $a\in\mathbb C(q)^\times$. If $V$ is an irreducible constituent of $W_q(\gb\lambda)$, then $V$ is an irreducible constituent of $W_q(\gb\mu)$.\hfill\qedsymbol
\end{prop}

An object $V\in\wcal C_\xi$ is said to have elliptic character $\gamma\in\widetilde\Gamma_\xi$ if  $V_{\gb\mu}\ne 0$ implies $\gamma_{_\xi}(\gb\mu)=\gamma$. Denote by $\wcal C_\xi^{\gamma}$ the abelian subcategory  of $\wcal C_\xi$ consisting of representations with elliptic character $\gamma$. The main result of this section is the following theorem.

\begin{thm}\label{t:blocks}
The categories $\wcal C_\xi^\gamma,\gamma\in\widetilde\Gamma_\xi$, are the blocks of $\wcal C_\xi$.
\end{thm}

This theorem was first proved in \cite{em} in the case that $\xi\in\mathbb C^\times$ satisfies $|\xi|\ne 1$ using analytic properties of the action of the $R$-matrix of \ust_\xi g. An $R$-matrix free approach was given  in \cite[\S8]{cmq} for the case that $\xi$ is not a root of unity. For $\xi=1$, the theorem was proved in \cite{cms}. We now give a proof that works for any $\xi\in\mathbb C'\backslash\{q\}$, thus completing the proof of Theorem \ref{t:blocks}.

For a brief review of the theory of blocks of an abelian category see \cite[\S1]{em}. The proof of Theorem \ref{t:blocks} is immediate from the following two propositions.

\begin{prop}\label{p:indechasec}
If $V\in\wcal C_\xi$ is indecomposable, then $V\in\wcal C_\xi^\gamma$ for some $\gamma\in\widetilde\Gamma_\xi$. In other words, $\wcal C_\xi = \opl_{\gamma\in\widetilde\Gamma_\xi}^{} \wcal C_\xi^\gamma$.
\end{prop}

The proof of Proposition \ref{p:indechasec} in the case $\xi=\zeta$ is analogous to that given in \cite{cms,cmq} and we omit the details (see \cite[\S8.4]{cmq}).

\begin{prop}\label{p:sameecsameb}
For every $\gamma\in\widetilde\Gamma_\xi$, the category $\wcal C_\xi^\gamma$ is an indecomposable abelian category.
\end{prop}

We now prove Proposition \ref{p:sameecsameb} for $\xi\ne q$. It suffices to show that, given $\gb\lambda,\gb\mu\in\cal P^+$ such that $\gamma_{_\xi}(\gb\lambda)=\gamma_{_\xi}(\gb\mu)$, there exists a sequence of indecomposable modules $W_1,\dots, W_m$ satisfying:
\begin{enumerate}
\item $V_\xi(\gb\lambda)$ is a constituent of $W_1$ and $V_\xi(\gb\mu)$ is a constituent of $W_m$,
\item $W_j$ and $W_{j+1}$ have at least one irreducible constituent in common.
\end{enumerate}
By Theorem \ref{t:lgenerator} and Proposition \ref{p:constspec}, it suffices to prove this in the case $\gb\lambda,\gb\mu\in\cal P_{\xi,I_\bullet}^+$. Thus, let $\gb\omega_1,\dots,\gb\omega_m\in\cal P_\xi^+$ be a sequence as in Lemma \ref{l:samebseq}. If $\xi=q$, it follows from Proposition \ref{p:trivconst} that the sequence $W_j=W_q(\gb\omega_j)$ satisfies the desired properties. If $\xi\in\mathbb C^\times$, set $\gb\omega_1'=\gb\lambda$ and define $\gb\omega_j'\in\cal P_q^+, j>1$, recursively by
\begin{equation}
\gb\omega'_{j+1} = \gb\omega_j'\left(\gb\tau_{q,k_j,a_j}\right)^{-\epsilon_j}.
\end{equation}
Set $\gb\mu'=\gb\omega_m', W_j' = W_q(\gb\omega_j')$ and $W_j=\overline{W_j'}$. By Proposition \ref{p:trivconst}, $W_j'$ and $W'_{j+1}$ have common irreducible constituents and, hence, so do $W_j$ and $W_{j+1}$ by Proposition \ref{p:constspec}. Since $W_1$ is a quotient of $W_\xi(\gb\lambda)$ and $W_m$ is a quotient of $W_\xi(\gb\mu)$ we are done.

In the case $\xi=1$, the proof of Proposition \ref{p:sameecsameb} we gave above is an alternative one to that given in \cite{cms}.

\subsection{On the q-characters of fundamental modules}\label{ss:charfund}

We now verify that the main result of \cite[\S6]{cmq} holds in the roots of unity as well. Throughout this subsection we assume $\xi\ne 1$. The next lemma is well-known and can be easily checked using the formulas of Lemma \ref{l:sl2actev} (cf. \cite[Proposition 9.2]{lroot}).

\begin{prop}\label{p:w=v}
Let $\lie g=\lie{sl}_2$ and $\lambda\in P^+_l$. Then $V_\zeta(\lambda)\cong W_\zeta(\lambda)$.\hfill\qedsymbol
\end{prop}

\begin{prop}\label{p:lcsl2}
Let $\lie g=\lie{sl}_2, I=\{i\}, a\in \mathbb C(\xi)^\times,\lambda\in P^+$, and $r=\lambda(h_i)$. Then,
$$\ch_\ell(W_\xi(a,\lambda))=\gb\omega_{i,a,r}\sum_{k=0}^r \left(\prod_{j=1}^k \gb\omega_{i,a\xi^{r-2j},2}\right)^{-1} = \gb\omega_{i,a,r}\sum_{k=0}^r \left(\prod_{j=1}^k \gb\alpha_{i,a\xi^{r-2j+1}}\right)^{-1}.$$
\end{prop}

\begin{proof}
For $\xi$ not a root of unity this was proved in \cite{fr} (see also \cite[Proposition 5.11]{cmq}). The root of unity case then follows follows from Proposition \ref{p:lcharspec} (see also \cite{fmroot}).
\end{proof}

\begin{lem}\label{l:bginv}
Let $V\in\wcal C_\xi$, $\gb\mu\in\cal P_\xi$, and  suppose that there exist a nonzero $v\in V_\gb\mu$  and $j\in I$ such that $x_{j,s}^+ v =0$ for all $s\in\mathbb Z$. Then, $\gb\mu_j(u)$ is a polynomial of degree $\wt(\gb\mu)(h_j)$ and $(x_{j,0}^-)^{\wt(\gb\mu)(h_j)}v\in V_{T_j\gb\mu}\backslash\{0\}$. Also, if the $\xi_j$ factorization of $\mu_j$ is given by
$$\mu_j = \prod_{r=1}^{k}\gb\omega_{j,a_r,m_r},$$
then $\gb\mu(\gb\alpha_{j,a_r\xi^{m_r-1}})^{-1}\in\wtl(V)$ for all $1\le r\le k$. Furthermore, for all $s\in\mathbb Z$,  we have
$$x_{j,s}^-v\in \sum_{r=1}^k \sum_{p=0}^{m_r-1}V_{\gb\mu\gb\alpha_{j,a_r\xi^{m_r-1-2p}}^{-1}},$$
and
$$\dim(V_{\gb\mu\gb\alpha_{j,a_r\xi^{m_r-1}}^{-1}})\ge \#\{1\le s\le k:a_r=a_s\}.$$
\end{lem}

\begin{proof}
The first two statements have already been proved. The remaining statements for $\xi$ of infinite order were proved in \cite[Proposition 6.5]{cmq}. In the root of unity case the proof can be carried out exactly as in the infinite order case after observing the following. Due to Lemma \ref{l:hinres} we can work with elements $\tilde h_{i,r}$ as it was done in \cite[Proposition 6.5]{cmq}. Moreover, observe as well that $\frac{[rc_{ij}]_{q_j}}{[r]_{q_i}}\in\mathbb A$ for all $i,j\in I$ and all nonzero integer $r$. Hence, we have the following identity in \uqtres g
$$(q_j^r+q_j^{-r})\ [\tilde h_{i,r},(x_{j,s}^-)^{(k)}] = \frac{[rc_{ij}]_{q_j}}{[r]_{q_i}}\ [\tilde h_{j,r},(x^-_{j,s})^{(k)}]$$
Finally, observe that $\zeta_j^r+\zeta_j^{-r}\ne 0$ for all $r\in\mathbb Z$ and set
$$\{rc_{ij}\}_\zeta = \epsilon_\zeta\left(\frac{[rc_{ij}]_{q_j}}{[r]_{q_i}}\right).$$
Hence, we get the following identity in \ustres g
\begin{equation}
[\tilde h_{i,r},(x_{j,s}^-)^{(k)}] = (\zeta_j^r+\zeta_j^{-r})^{-1}\{rc_{ij}\}_\zeta\ [\tilde h_{j,r},(x^-_{j,s})^{(k)}]
\end{equation}
which replaces equation (6.19) of \cite{cmq}. From here, all the steps of the proof can be performed exactly as in the case of $\xi$ generic.
\end{proof}

Given $\lambda\in P^+$, let $I(\lambda)=\{i\in I:\lambda(h_i)=0\}$ and let $\cal W(\lambda)$ be the subgroup of $\cal W$ generated by $\{s_i: i\in I(\lambda)\}$.
The proof of the following lemma can be found in \cite[\S1.10]{hum:coxeter}.

\begin{lem}Let $\lambda\in P^+$. Then,
\begin{enumerate}
\item $\cal W(\lambda)=\{w\in W: w\lambda=\lambda\}$.
\item Each left coset of $\cal W(\lambda)$ in $\cal W$ contains a unique element of minimal length.
\item Denote by $\cal W_\lambda$ the set of all left coset representatives of minimal length, suppose that $w\in \cal W_\lambda$, and that $w=s_jw'$ for some $w'\in\cal W$ with $\ell(w')=\ell(w)-1$. Then, $w'\in \cal W_\lambda$. \hfill\qedsymbol
\end{enumerate}
\end{lem}

The following theorem can now be proved exactly as in \cite[Theorem 6.1]{cmq}.

\begin{thm}\label{t:bginv}
Suppose $\lie g$ is of classical type. Let $i\in I$, $a\in\mathbb C(\xi)^\times$ and assume that $\gb\lambda\in\wt_\ell(V_\xi(\gb\omega_{i,a}))$ is such that $\wt(\gb\lambda)=\lambda\in P^+$. Then,
$$\dim(V_\xi(\gb\omega_{i,a})_\gb\lambda)=\dim(V_\xi(\gb\omega_{i,a})_{T_w\gb\lambda}) \qquad\text{and}\qquad T_w(\wt_\ell(V_\xi(\gb\omega_{i,a})_\lambda))=\wt_\ell(V_\xi(\gb\omega_{i,a})_{w\lambda})$$
for all $w\in\cal W_\lambda$. Suppose further that $\gb\lambda\ne \gb\omega_{i,a}$. Then, there exist
$\gb\mu\in\wt_\ell(V_\xi(\gb\omega_{i,a})), b,c\in\mathbb C(\xi)^\times$, and $j\in I$ such that $\gb\mu_j(u)=(1-bu)(1-cu)$, and
\begin{equation}\label{e:indom}
\gb\lambda=\gb\mu(\gb\alpha_{j,b})^{-1}.
\end{equation}
Moreover, if $c\ne b\xi^{-2}$, then $\gb\mu\ (\gb\alpha_{j,c})^{-1}\in\wtl(V_\xi(\gb\omega_{i,a}))$  and,  if $c=b$, then
$\dim(V_\xi(\gb\omega_{i,a})_{\gb\lambda})\ge 2$.\hfill\qedsymbol
\end{thm}

Theorem \ref{t:bginv} provides an algorithm for computing a lower bound for $\chl(V_\xi(\gb\omega_{i,a}))$. For $\xi$ of infinite order, it was used together with the knowledge of $\ch(V_\xi(\gb\omega_{i,a}))$ (see \cite{cha:fer}) in \cite{cm:fund} to compute $\chl(V_\xi(\gb\omega_{i,a}))$. Notice that it follows from Proposition \ref{p:lcharspec} that $\epsilon_\xi(\chl(V_q(\gb\omega_{i,a})))$ is an upper bound for $\chl(V_\xi(\gb\omega_{i,a}))$.
We finish the paper by giving an example explaining how to use these facts and Corollary \ref{c:nodom} to identify values of $\xi$ for which
\begin{equation}\label{e:specir}
V_\xi(\gb\omega_{i,a})\cong \overline{V_q(\gb\omega_{i,a})}.
\end{equation}
It follows from Proposition \ref{p:lcharspec} that \eqref{e:specir} holds iff
\begin{equation}\label{e:ecf}
\chl(V_\xi(\gb\omega_{i,a}))=\epsilon_\xi(\chl(V_q(\gb\omega_{i,a}))).
\end{equation}

\begin{rem}
Observe that, if $\omega_i$ is minuscule, then \eqref{e:specir} holds for any value of $\xi$ (including $\xi=1$). In fact, given $a\in\mathbb C(\xi)$ and a minuscule $\omega_i$, we have
\begin{equation*}
\chl(V_\xi(\gb\omega_{i,a})) = \sum_{w\in\cal W_{\omega_i}} T_w(\gb\omega_{i,a}).
\end{equation*}
On the other hand, if $\omega_i$ is not minuscule and $\xi=1$, then \eqref{e:specir} is always false because $V(\gb\omega_{i,a})$ is irreducible as $\lie g$-module, while $\overline{V_q(\gb\omega_{i,a})}$ is not.
\end{rem}

\begin{ex}\label{ex:dn}
Assume from now on that $\xi=\zeta, a\in\mathbb C^\times$, $\lie g$ is of type $D_n$ and $i=2$. Let $\gb\mu_j\in\cal P_q$ be defined by
\begin{equation}
\gb\mu_j   =
\begin{cases}
(\gb\omega_{j-1,aq^{j+1}})^{-1} \gb\omega_{j-1,aq^{2n-j-3}}\gb\omega_{j,aq^{j}}(\gb\omega_{j,aq^{2n-j-2}})^{-1},& \text{ if } 1\leq j \leq n-2,\\
\gb\omega_{j,aq^{n-3}}(\gb\omega_{j,aq^{n+1}})^{-1}, & \text{ if } j=n-1,n.
\end{cases}
\end{equation}
Thus, according to \cite[(6.8)]{cmq},
\begin{equation}
\chl(V_q(\gb\omega_{2,a})) = \sum_{w\in\cal W_{\omega_2}} T_w(\gb\omega_{2,a}) + \sum_{j\ne n-2}\gb\mu_j +2\gb\mu_{n-2}.
\end{equation}
Since $\wt(T_w(\gb\omega_{2,a}))\notin P^+$, by Corollary \ref{c:nodom}, \eqref{e:specir} holds provided $\bgb\mu_j\notin\cal P^+$ for all $j=1,\dots, n$. If $1<j<n-1$, $\bgb\mu_j\in\cal P^+$ iff $l$ divides $2(n-j-2)$ and $2(n-j-1)$ which implies $l$ divides 2. Similarly, $\bgb\mu_{m-1},\bgb\mu_n\in\cal P^+$ iff $l$ divides 4. Since we are assuming $l$ is odd, we conclude $\bgb\mu_j\notin \cal P^+$ for all $j>1$. Finally, $\bgb\mu_1\in\cal P^+$ iff $l$ divides $2(n-2)$. Hence, since $l$ is odd, if $l$ is not a divisor of $n-2$, \eqref{e:specir} holds. We now compute the multiplicity of the space $V(\gb\omega_{2,a})_{\bgb\mu_j}$ assuming $l$ does not divide $n-2$. Notice that
\begin{equation}\label{e:multdnreg}
\bgb\mu_j = \bgb\mu_k \quad\text{iff}\quad 1\le j=k-1<n-2 \quad\text{and}\quad l\quad\text{divides}\quad n-2-j.
\end{equation}
Hence, $\bgb\mu_{n-3}\ne\bgb\mu_{n-2}$ for any value of $l$ and
\begin{equation}\label{e:multdn2}
\dim(V_\zeta(\gb\omega_{2,a})_{\bgb\mu_j}) =
\begin{cases}
2,& \quad\text{if } j=n-2 \text{ or if } j< n-2 \text{ and } l\text{ divides } n-2-j,\\
1,& \quad\text{otherwise.}
\end{cases}
\end{equation}

Let us now discuss the case when $l$ divides $n-2$. As pointed  out to us by Nakajima, it follows from the algorithm given in \cite{nak:qvtqchar} that
\begin{equation}\label{e:dnlcir}
\chl(V_\zeta(\gb\omega_{2,a})) = \sum_{w\in\cal W_{\omega_2}} T_w(\gb\omega_{2,a}) + \sum_{j\ne 1,n-2}\bgb\mu_j +2\bgb\mu_{n-2}.
\end{equation}
Moreover, $\dim(V_\zeta(\gb\omega_{2,a})_{\bgb\mu_j})$ for $j\ne 1$ is given by \eqref{e:multdn2}.
We now give an alternate proof of \eqref{e:dnlcir} using Theorem \ref{t:bginv} and the theory of specialization of modules. Let $\xi$ be any element of $\mathbb C'$ again. As in \cite[\S6.5]{cmq}, set
$$w_j = s_{j-1}\cdots s_1s_{j+1}\cdots s_{n-2} s_{n}s_{n-1} s_{n-2}\cdots s_1.$$
Observe that  $\cal W_{\omega_2} = \{w_j:j=1,\dots,n\}$ and $w_j\omega_2= \alpha_j$. One then computes (cf. \cite[(6.10)]{cmq})
\begin{equation}
T_{w_j}(\gb\omega_{2,a}) =
\begin{cases}
(\gb\omega_{j-1,a\xi^{j+1}})^{-1} \gb\omega_{j,a\xi^{j}} \gb\omega_{j,a\xi^{2n-4-j}} (\gb\omega_{j+1,a\xi^{2n-3-j}})^{-1},& \ \text{ if } j\leq n-2,\\
(\gb\omega_{n-2,a\xi^{n}})^{-1} \gb\omega_{j,a\xi^{n-1}}\gb\omega_{j,a\xi^{n-3}},&\ \text{ if } j=n-1,n.
\end{cases}
\end{equation}
Notice
\begin{equation}
\bgb\mu_j =
\begin{cases}
T_{w_j}(\gb\omega_{2,a})(\gb\alpha_{j,a\xi^{2n-4-j}})^{-1},& \text{ if } 1\leq j \leq n-2,\\
T_{w_{j-1}}(\gb\omega_{2,a})(\gb\alpha_{j-1,a\xi^{j-1}})^{-1},& \text{ if } 2\leq j \leq n-2,\\
T_{w_j}(\gb\omega_{2,a})(\gb\alpha_{i,a\xi^{n-3}})^{-1}, & \text{ if } j=n-2, i=n-1,n,\\
T_{w_j}(\gb\omega_{2,a})(\gb\alpha_{j,a\xi^{n-1}})^{-1}, & \text{ if } j=n-1,n.
\end{cases}
\end{equation}
It immediately follows from Theorem \ref{t:bginv} that $\bgb\mu_j\in\wtl(V_\xi(\gb\omega_{2,a}))$ for $j\ge n-2$ and, moreover, that $\dim(V_\xi(\gb\omega_{2,a})_{\bgb\mu_{n-2}})\ge 2$. Let $1<j\le n-2$ and observe that
\begin{equation}
\frac{\zeta^{2n-4-j}}{\zeta^j} = \zeta^{-2} \quad\Rightarrow\quad \frac{\zeta^{j-1}}{\zeta^{2n-3-j}}\ne \zeta^{-2}.
\end{equation}
Together with Theorem \ref{t:bginv}, this implies that, if $1<j\le n-2$, either $T_{w_j}(\gb\omega_{2,a})(\gb\alpha_{j,a\xi^{2n-4-j}})^{-1}\in\wtl(V_\zeta(\gb\omega_{2,a}))$ or $T_{w_{j-1}}(\gb\omega_{2,a})(\gb\alpha_{j-1,a\xi^{j-1}})^{-1}\in\wtl(V_\zeta(\gb\omega_{2,a}))$. Hence, $\bgb\mu_j\in\wtl(V_\zeta(\gb\omega_{2,a}))$ for all $j>1$. It remains to show that $\bgb\mu_1\notin\wtl(V_\zeta(\gb\omega_{2,a}))$ if $l$ divides $n-2$.

To do this, let $V=V_q(\gb\omega_{2,a})$, $v$ be a highest-$\ell$-weight vector of $V$, $L=\uqtres g v$, and $v_{w_j}$ be as in \eqref{e:v_w}. In particular, $v_{w_j}$ is an $\mathbb A$-basis element of $L$ and, hence, its image $\overline{v_{w_j}}\in\overline V$ is nonzero. Let also $W_j=\uqt{g_j}v_{w_j}$ and $\overline{W_j}=\ust_\zeta{g_j}\overline{v_{w_j}}$. Since $v_{w_j}$ is a highest $\ell$-weight vector for \uqt{g_j}, $W_j$ is a quotient of the Weyl module for this subalgebra with highest $\ell$-weight $(T_{w_j}(\gb\omega_{2,a}))_j$ (and similarly for $\overline{W_j}$). It follows from the above computations that $V_{\gb\mu_j}\subseteq W_j$ and that $V_0 = \sum_{j=1}^{n} (W_j)_0$. We claim that $W_1$ must be a Weyl module. In fact, it is easy to see that $V_{\gb\mu_1}\cap W_j\ne \{0\}$ iff $j=1$. If $W_1$ were not a Weyl module, then it would be irreducible and so would be $\overline{W_1}$. This would imply that $(\bgb\mu_1)_1$ (which is the the constant polynomial) would be an $\ell$-weight of the irreducible \ust_\zeta{{sl}_2}-module with highest $\ell$-weight $(1-a\zeta u)(1-a\zeta^{-1}u)$. This contradicts Proposition \ref{p:lcsl2}.

Since $v_{w_1}$ is an $\mathbb A$-basis element of $L$, it follows that $\overline{W_1}$ is also a Weyl module (alternatively, if $\overline{W_1}$ were not a Weyl module, $\overline V$ would not be highest-$\ell$-weight since $\overline V_{\bgb\mu_1}\cap \overline{W_j}=\{0\}$ if $j>1$ and we would have a contradiction again). Therefore, there exists $v'\in(\overline{W_1})_0$ which generates a trivial submodule of $\overline{W_1}$. It is easy to see that $v'$ also generates a trivial submodule of $\overline V$. This shows that $\bgb\mu_1$ is not an $\ell$-weight of the irreducible quotient of $\overline V$ which is isomorphic to $V_\zeta(\gb\omega_{2,a})$.
\end{ex}

\begin{rem}
The assumption that $l$ is odd is not essential above and one can easily deduce similar results for roots of unity of even order as well.  Notice that if $l$ divides $n-2$, then $W_{\zeta}(\gb\omega_{2,a})$ is reducible and, hence, not a tensor product of fundamental representations.
\end{rem}

We believe that the above line of reasoning can be used alongside the results of \cite{cm:fund} to obtain expressions for the q-characters of the fundamental representations of \ust_\zeta g in terms of the braid group action for any $\lie g$ of classical type.  In order to keep the length of present text within reasonable limits, we postpone further discussion in this direction to a forthcoming publication.

\vspace{5pt}
\noindent{\bf Acknowledgements:}  D.J. thanks FAPESP for the financial support for her visit to the State University of Campinas during which this paper was developed and the Mathematics Department of the State University of Campinas for hospitality.  The work of A.M. was partially supported by CNPq and FAPESP. D.J. also thanks Anna Lachowska for motivating discussions at the initial stage of this paper. We thank Hiraku Nakajima for several references, pointing out a mistake in \S\ref{ss:charfund} in the first draft of the paper, and very kind and helpful communication. We also thank the referee for constructive comments and, in particular, for a question that lead us to strengthen the results of Section \ref{s:hlwtensor}.

\bibliographystyle{amsplain}

\end{document}